%% file: t-lieb2.tex
\documentclass[12pt,a4paper]{amsart}
\usepackage{latexsym}
\usepackage{youngtab}           % to get LASY symbols
\usepackage{graphicx}             % to insert PostScript figures
\usepackage{psfig,color}
\usepackage{latexsym}
\usepackage{amsfonts}
\usepackage{amsxtra}
\usepackage{amsmath}
\usepackage{amssymb}
\usepackage{mathrsfs}
\usepackage{amsthm}
\usepackage{fullpage}

\theoremstyle{plain}
\newtheorem{theorem}{Theorem}[section]

\newtheorem{lemma}[theorem]{Lemma}
\newtheorem{corollary}[theorem]{Corollary}

\theoremstyle{definition}

\newtheorem{example}{Example}[section]
\newtheorem{remark}{Remark}[section]
\numberwithin{equation}{section}

\DeclareMathOperator{\SHAPE}{Shape}

\input xy
\xyoption{all}

\newcommand{\nc}{\newcommand}
\nc{\dten}{10} \nc{\deleven}{11} \nc{\dtwelve}{12}
\nc{\dthirteen}{13} \nc{\dfourteen}{14} \nc{\dfifteen}{15}
\nc{\dsixteen}{16}\nc{\dseventeen}{17}

\begin{document}
\title[Temperley--Lieb Algebras]
{Representations of Temperley--Lieb Algebras}

\author[J. Enyang]{John Enyang}

%\subjclass{Primary 47A15; Secondary 46A32, 47D20}

\keywords{Temperley--Lieb algebra; Specht module; cellular algebra; Murphy basis; Jucys--Murphy operators}

\begin{abstract}
We define a commuting family of operators $T_0,T_1,\dots,T_n$ in the Temperley--Lieb algebra $\mathcal{A}_n(x)$ of type $A_{n-1}$. Using an appropriate analogue to Murphy basis of the Iwahori--Hecke algebra of the symmetric group, we describe the eigenvalues arising from the triangular action of the said operators on the cell modules of $\mathcal{A}_n(x)$.  These results are used to provide the Temperley--Lieb algebras of type $A_{n-1}$ with a semi--normal form, together with a branching law, and explicit formulae for associated Gram determinants.
\end{abstract}
%\acknowledgements{The author is grateful to Arun Ram for communicating the results of~\cite{hmram:2007}.}
\maketitle

\section{The Temperley--Lieb Algebras}
Let $n$ be a non--negative integer, $x$ be an indeterminate over $\mathbb{Z}$ and write $R=\mathbb{Z}[x]$. The Temperley--Lieb algebra $\mathcal{A}_n(x)$, defined in~\cite{temperley-lieb}, is the unital associative $R$--algebra generated by $e_1,\dots,e_{n-1}$ which are subject to the defining relations
\begin{align}
&e_i^2=xe_i,&&\text{for $i=1,\dots,n-1$;}\label{defrel:1}\\
&e_ie_{i\pm1}e_i=e_i,&&\text{for $i,i\pm1=1,\dots,n-1;$}\label{defrel:2}\\
&e_ie_j=e_je_i,&&\text{for $i,j=1,\dots,n-1$ and $|i-j|\ge 2$.}\label{defrel:3}
\end{align}
By convention, $\mathcal{A}_1(x)=R$ and, for $i=2,3,\dots$, we regard $\mathcal{A}_{i}(x)$ as the subalgebra of $\mathcal{A}_{i+1}(x)$ generated by $e_1,\dots,e_{i-1}$, giving a tower
\begin{align}
\mathcal{A}_1(x)\subseteq\mathcal{A}_2(x)\subseteq\mathcal{A}_3(x)\subseteq\cdots.\label{tower}
\end{align}
Using restriction in the tower~\eqref{tower}, we construct cellular bases, in the sense of~\cite{grahamlehrer}, for $\mathcal{A}_n(x)$ which are compatible with the action of certain commuting operators in $\mathcal{A}_n(x)$.
\section{Murphy Bases for the Temperley--Lieb Algebras}
For the purposes of these notes, a \emph{partition} of $n$ is a pair of integers $\lambda=(i,n-2i)$, where $0\le 2i\le n$. If $\lambda=(i,n-2i)$ and $\mu=(j,n-2j)$ are partitions of $n$, we will write $\lambda\unrhd\mu$ if $i\ge j$, while $\lambda\rhd\mu$ will signify that $\lambda\unrhd\mu$ and $\lambda\ne\mu$. If $\mu$ is a partition of $n-1$ and $\lambda=(i,n-2i)$ is a partition of $n$, write $\mu\to\lambda$ if $\mu=(i,n-2i-1)$ or $\mu=(i-1,n-2i+1)$. Let $\lambda$ be a partition of $n$. Define 
\begin{align*}
\mathfrak{T}_n(\lambda)=\{(\lambda^{(0)},\lambda^{(1)},\dots,\lambda^{(n)}):\text{$\lambda^{(0)}=(0,0)$, $\lambda^{(k-1)}\to\lambda^{(k)}$ for $k=1,\dots,n$, and $\lambda^{(n)}=\lambda$}\}.
\end{align*}
If $\mathfrak{s}=(\lambda^{(0)},\dots,\lambda^{(n)})\in\mathfrak{T}_n(\lambda)$, write $\lambda=\SHAPE(\mathfrak{s})$ and let $\mathfrak{s}|_k=(\lambda^{(0)},\dots,\lambda^{(k)})$, for $k=1,\dots,n$. 
We order the elements of $\mathfrak{T}_n(\lambda)$ by writing $\mathfrak{s}\unrhd\mathfrak{t}$ if $\SHAPE(\mathfrak{s}|_k)\unrhd\SHAPE(\mathfrak{t}|_k)$ for $k=1,\dots,n$, and $\mathfrak{s},\mathfrak{t}\in\mathfrak{T}_n(\lambda)$; by $\mathfrak{s}\rhd\mathfrak{t}$ we will mean that $\mathfrak{s}\unrhd\mathfrak{t}$ and $\mathfrak{s}\ne\mathfrak{t}$. If $\mathfrak{t}=(\lambda^{(0)},\dots,\lambda^{(n)})$, then $\mathfrak{t}$ may be identified with an up--down tableau: 
\begin{align*}
\mathfrak{t}\mapsto(\mathfrak{t}^{(0)},\dots,\mathfrak{t}^{(n)}),&&\text{where $\mathfrak{t}^{(k)}=\underbrace{\text{\tiny\Yvcentermath1$\yng(2)$}\cdots\text{\tiny\Yvcentermath1$\yng(1)$}}_{\text{$(k-2i)$ boxes}}$ whenever $\lambda^{(k)}=(i,k-2i)$, for $k=0,\dots,n$.}
\end{align*}
In turn, the up--down tableaux correspond to paths in the Bratteli diagram associated with the Temperley--Lieb algebras (\emph{cf.} \S2 of~\cite{ramhalv:jb}).

If $f$ is an integer, $0\le f\le [n/2]$ and $\lambda=(f,{n-2f})$, define 
\begin{align*}
m_\lambda=e_1e_3\cdots e_{2f+1}
\end{align*}
and let $\mathcal{A}_n^\lambda$ denote the two sided ideal in $\mathcal{A}_n(x)$ generated by $m_\lambda$ and 
\begin{align*}
\check{\mathcal{A}}_n^\lambda=\sum_{\mu\rhd\lambda}\mathcal{A}_n^\mu.
\end{align*}
If $n=2k+\delta$, where $\delta\in\{0,1\}$, then 
\begin{align*}
0\subset\mathcal{A}_n^{(k,n-2k)}\subset\mathcal{A}_n^{(k-1,n-2k+2)}\subset\cdots\subset\mathcal{A}_n^{(0,n)}=\mathcal{A}_n(x)
\end{align*}
is a filtration by two sided ideals of $\mathcal{A}_n(x)$.

If $i,j=1,2,\dots,n$, define $w_{i,j}\in \mathcal{A}_n$ by 
\begin{align*}
w_{i,j}=
\begin{cases}
e_ie_{i+1}\cdots e_{j-1},&\text{if $i<j$;}\\
e_{i-1}e_{i-2}\cdots e_j, &\text{if $j<i$;}\\
1, &\text{otherwise.}
\end{cases}
\end{align*}
Now, introduce elements 
\begin{align*}
\{v_\mathfrak{t}: \text{$\mathfrak{t}\in\mathfrak{T}_n(\lambda)$, $\lambda$ a partition of $n$}\},
\end{align*}
by writing $v_{\mathfrak{t}}=1$ if $\mathfrak{t}=((0,0))$ and, otherwise, if $\mathfrak{t}\in\mathfrak{T}_n(\lambda)$, where $\lambda=(f,n-2f)$, and $\mathfrak{s}=\mathfrak{t}|_{n-1}$, then 
\begin{align*}
v_\mathfrak{t}=
\begin{cases}
v_\mathfrak{s},&\text{if $\SHAPE(\mathfrak{s})=(f,n-2f-1)$}\\
w_{2f,n}v_\mathfrak{s},&\text{if $\SHAPE(\mathfrak{s})=(f-1,n-2f+1)$.}
\end{cases}
\end{align*}
Similarly, we define 
\begin{align*}
\{v_\mathfrak{t}^*: \text{$\mathfrak{t}\in\mathfrak{T}_n(\lambda)$, $\lambda$ a partition of $n$}\},
\end{align*}
by writing $v_{\mathfrak{t}}^*=1$ if $\mathfrak{t}=((0,0))$, and, otherwise, if $\mathfrak{t}\in\mathfrak{T}_n(\lambda)$, where $\lambda=(f,n-2f)$, and $\mathfrak{s}=\mathfrak{t}|_{n-1}$, then 
\begin{align*}
v_\mathfrak{t}^*=
\begin{cases}
v_\mathfrak{s}^*,&\text{if $\SHAPE(\mathfrak{s})=(f,n-2f-1)$}\\
v_\mathfrak{s}^*w_{n,2f},&\text{if $\SHAPE(\mathfrak{s})=(f-1,n-2f+1)$.}
\end{cases}
\end{align*}
The following are stated for reference in subsequent calculations.
\begin{lemma}
Suppose that $\lambda=(f,n-2f)$ is a partition of $n$, with $f>0$ and $n-2f\ge 1$. Let $\mathfrak{s},\mathfrak{t}\in\mathfrak{T}_n(\lambda)$ satisfy $\mathfrak{s}|_{n-2}=\mathfrak{t}|_{n-2}$ and $\mathfrak{s}\ne\mathfrak{t}$. Then the condition
\begin{align*}%\label{cont:a}
&\SHAPE(\mathfrak{s}|_{n-1})=(f-1,n-2f+1),&&\text{and} &&\SHAPE(\mathfrak{t}|_{n-1})=(f,n-2f-1)
\end{align*}
holds if and only if $v_\mathfrak{s}=v_\mathfrak{t}e_{n-1}$.
\end{lemma}
\begin{proof}
Suppose that $\mathfrak{s}\ne\mathfrak{t}$ and $\mathfrak{s}|_{n-2}=\mathfrak{t}|_{n-2}$ and consider the Bratteli diagram fragment
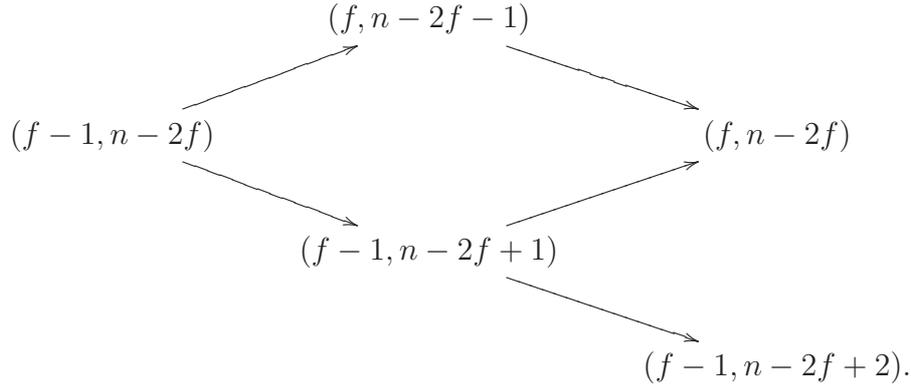
\begin{align*}
\xymatrix{
				&(f,n-2f-1)\ar[dr]		&	\\			
(f-1,n-2f)\ar[ur]\ar[dr]	& 				& (f,n-2f)\\
				&(f-1,n-2f+1)\ar[ur]\ar[dr]	& \\
				&				& (f-1,n-2f+2).
}
\end{align*}
If $\mathfrak{u}=\mathfrak{t}|_{n-2}$, then either,
\begin{align*}
\SHAPE(\mathfrak{s}|_{n-1})=(f,n-2f-1)&&\text{and}&& \SHAPE(\mathfrak{t}|_{n-1})=(f-1,n-2f+1),
\end{align*}
in which case $v_{\mathfrak{s}}=w_{2f,n-1}v_\mathfrak{u}$ and $v_\mathfrak{t}=w_{2f,n}v_\mathfrak{u}$, or 
\begin{align*}
\SHAPE(\mathfrak{s}|_{n-1})=(f-1,n-2f+1)&&\text{and}&& \SHAPE(\mathfrak{t}|_{n-1})=(f,n-2f-1),
\end{align*}
in which case $v_{\mathfrak{s}}=w_{2f,n}v_\mathfrak{u}$ and $v_\mathfrak{t}=w_{2f,n-1}v_\mathfrak{u}$. Since $w_{2f,n}=w_{2f,n-1}e_{n-1}$, and $e_{n-1}$ commutes with $v_\mathfrak{u}$, the result follows.  
\end{proof}
\begin{corollary}\label{restr}
Let $\lambda$ be a partition of $n$ and suppose that $k$ is an integer, $1< k<n$. If $\mathfrak{s},\mathfrak{t}\in\mathfrak{T}_n(\lambda)$ satisfy $\mathfrak{s}\ne\mathfrak{t}$ and $\SHAPE(\mathfrak{s}|_i)=\SHAPE(\mathfrak{t}|_i)$, for $i\in\{0,1,\dots,n\}\setminus\{k-1\}$, then the condition
\begin{align*}%\label{cont:b}
&\SHAPE(\mathfrak{s}|_{k-1})=(j-1,k-2j+1),&&\text{and} &&\SHAPE(\mathfrak{t}|_{k-1})=(j,k-2j-1)
\end{align*}
holds if and only if $v_\mathfrak{s}=v_\mathfrak{t}e_{k-1}$.
\end{corollary}

In~\cite{grahamlehrer}, J.~Graham and G.~Lehrer have demonstrated that $\mathcal{A}_n(x)$ is cellular, while M.~H\"arterich has provided certain Murphy type bases for generalised Temperley--Lieb algebras in~\cite{harterich}. In order to obtain a triangular action for the commuting family of elements defined in \S\ref{jmelements}, we establish that $\mathcal{A}_n(x)$ has a cellular basis as described in the next lemma (\emph{cf.} Example~\ref{ex:mbasis} below).   
\begin{lemma}\label{saru-templ}
The algebra $\mathcal{A}_{n}(x)$ is freely generated as an $R$--module by
the collection
\begin{align}\label{r-gen}
\{m_{\mathfrak{uv}}=v_\mathfrak{u}^*m_\lambda v_\mathfrak{v}:\text{for $\mathfrak{u},\mathfrak{v}\in\mathfrak{T}_{n}(\lambda)$ and $\lambda$ a partition of $n$}\}.
\end{align}
Moreover, the following statements hold.
\begin{enumerate}
\item The $R$--linear map defined by $*:m_{\mathfrak{uv}}\mapsto m_{\mathfrak{vu}}$, for $\mathfrak{u},\mathfrak{v}\in\mathfrak{T}_n(\lambda)$ and $\lambda$ a partition of $n$, is the algebra anti--involution of $\mathcal{A}_n(x)$ satisfying $e_i\mapsto e_i$ for $i=1,\dots,n-1$.
\item  Suppose that $b\in \mathcal{A}_n(x)$ . If $\lambda$ is a partition of $n$, and $\mathfrak{u}\in\mathfrak{T}_{n}(\lambda)$, then there exist $a_{\mathfrak{v}}\in R$, for $\mathfrak{v}\in\mathfrak{T}_{n}(\lambda)$, such that\label{act:0}
\begin{align}\label{btsu:1}
m_{\mathfrak{su}} b \equiv\sum_{\mathfrak{v}} a_\mathfrak{v}m_{\mathfrak{sv}}\mod\check{\mathcal{A}}^\lambda_{n},
\end{align}
for all $\mathfrak{s}\in\mathfrak{T}_{n}(\lambda)$.
\end{enumerate}
\end{lemma}
Note that Lemma~\ref{saru-templ} implies that, if $\lambda$ is a partition of $n$, then $\check{\mathcal{A}}_n^\lambda$ is the $R$--module freely generated by the set $\{m_{\mathfrak{uv}}:\text {$\mathfrak{u},\mathfrak{v}\in \mathfrak{T}_{n}(\mu)$, for $\mu\rhd\lambda$}\}$. 

For $k$ an integer with $1\le 2k+1<n$, and $\mu=(k,n-2k)$, let 
\begin{align*}
\mathfrak{T}^{(k)}_n(\mu)=\{\mathfrak{s}\in\mathfrak{T}_{n}(\mu):v_\mathfrak{s}\in\langle e_{2k+1},\dots,e_{n-1}\rangle\}.
\end{align*} 
After observing that the map 
\begin{align*}
*:m_{\mathfrak{uv}}\mapsto m_\mathfrak{vu}, &&\text{for $\mathfrak{u},\mathfrak{v}\in\mathfrak{T}_n(\lambda)$, and $\lambda$ a partition of $n$,}
\end{align*}
coincides with the algebra anti--involution that fixes the set $\{e_i:i=1,\dots,n-1\}$ pointwise, Lemma~\ref{saru-templ} will follow from the following statement. 
\begin{lemma}\label{stronger}
The set $\{m_{\mathfrak{uv}}:\text{$\mathfrak{u},\mathfrak{v}\in\mathfrak{T}_{n}(\lambda)$ and $\lambda$ a partition of $n$}\}$ freely generates $\mathcal{A}_n(x)$ as an $R$--module. Moreover, if $b\in \mathcal{A}_n(x)$, $\lambda=(f,n-2f)$ is a partition, and $\mathfrak{u}\in\mathfrak{T}_{n}(\lambda)$, then there exist $a_{\mathfrak{v}}\in R$, for $\mathfrak{v}\in\mathfrak{T}_{n}(\lambda)$, which depend only on $\mathfrak{u}$, such that
\begin{align}\label{btsu:2}
m_\lambda v_{\mathfrak{u}} b=\sum_{\mathfrak{v}\in\mathfrak{T}_{n}(\lambda)} a_\mathfrak{v}m_\lambda v_{\mathfrak{v}}+\sum_{\substack{\mu\rhd\lambda\\\mathfrak{r},\mathfrak{t}\in\mathfrak{T}_n(\mu)}}a_{\mathfrak{rt}}m_{\mathfrak{rt}},
\end{align}
where the sum is over partitions $\mu=(k,n-2k)$, and $\mathfrak{r}\in\mathfrak{T}_n^{(f)}(\mu)$, with $k=f+1,f+2,\dots$, and $a_{\mathfrak{rt}}\in R$, for $\mathfrak{r}\in\mathfrak{T}_n^{(f)}(\mu)$ and $\mathfrak{t}\in\mathfrak{T}_n(\mu)$. 
\end{lemma}

\begin{lemma}\label{techs:1}
Let $\lambda=(f,n-2f)$, where $n>n-2f>0$. Write $\tau=(f+1,n-2f-2)$ and $\nu=(f,n-2f-1)$. If $\mathfrak{t}\in\mathfrak{T}_{n-1}^{(f-1)}(\nu)$ and $\mathfrak{u}\in\mathfrak{T}_{n-1}(\nu)$, then either 
\begin{align*}
e_{2f-1}w_{2f,n}v_\mathfrak{t}^*m_\nu v_\mathfrak{u}=m_\lambda v_\mathfrak{v}, &&\text{where $\mathfrak{v}\in\mathfrak{T}_n(\lambda)$ and $\mathfrak{v}|_{n-1}=\mathfrak{u}$,}
\end{align*}
or there exists $\mathfrak{s}\in\mathfrak{T}_n^{(f)}(\tau)$, such that
\begin{align*}
e_{2f-1}w_{2f,n}v_\mathfrak{t}^*m_\nu v_\mathfrak{u} =v_\mathfrak{s}^*m_\tau v_\mathfrak{v},&&\text{where $\mathfrak{v}\in\mathfrak{T}_{n}(\tau)$ and $\mathfrak{v}|_{n-1}=\mathfrak{u}$.}
\end{align*}
\end{lemma}
\begin{proof}
We may write $v_\mathfrak{t}^*=w_{j,2f}$, where $2f\le j\le n-1$, so that
\begin{align*}
e_{2f-1}w_{2f,n}v^*_\mathfrak{t}m_\nu&= e_{2f-1}w_{2f,n} w_{j,2f} m_\nu \\
&=
\begin{cases}
e_{2f-1}e_{2f}m_\nu,&\text{if $j=n-1$;}\\
e_{2f-1}e_{2f}e_{j+1}e_{j+2}\cdots e_{n-1}m_\nu,&\text{if $2f\le j<n-1$.}
\end{cases}
\end{align*}
In the first case in the above expression, we obtain 
\begin{align*}
e_{2f-1}w_{2f,n}v^*_\mathfrak{s}m_\nu v_\mathfrak{u}= e_{2f-1}e_{2f} e_1\cdots e_{2f-1}v_\mathfrak{u}=e_1\cdots e_{2f-1}v_\mathfrak{u}
=m_\lambda v_\mathfrak{v},
\end{align*}
whereas in the second,
\begin{align*}
e_{2f-1}w_{2f,n}v^*_\mathfrak{s}m_\nu v_\mathfrak{u}&=e_{2f-1}e_{2f}e_{j+1}e_{j+2}\cdots e_{n-1}(e_1\cdots e_{2f-1})v_\mathfrak{u}\\
&=e_1\cdots e_{2f-1} e_{j+1}e_{j+2}\cdots e_{n-1} v_\mathfrak{u}\\
&=e_{j+1}e_{j}\cdots e_{2f+2}m_\tau e_{2f+2}e_{2f+3}\cdots e_{n-1}v_\mathfrak{u}=v_\mathfrak{s}^*m_\tau v_\mathfrak{v},
\end{align*}
as required.
\end{proof}

If $\lambda$ is a partition of $n$, let $m_{\mathfrak{t}^\lambda}=m_\lambda+\check{\mathcal{A}}_{n}^\lambda\in \mathcal{A}_{n}^\lambda/\check{\mathcal{A}}_{n}^\lambda$, and define $C^\lambda$ to be the right $\mathcal{A}_{n}(x)$--submodule of $\mathcal{A}_{n}/\check{\mathcal{A}}_{n}^\lambda$ generated by $m_{\mathfrak{t}^\lambda}$.  
Further, if $\lambda=(f,n-2f)$ and $\mu\to\lambda$, define 
\begin{align*}
y^\lambda_\mu=
\begin{cases}
m_{\mathfrak{t}^\lambda},&\text{if $\mu=(f-1,n-2f+1)$;}\\
m_{\mathfrak{t}^\lambda}w_{2f,n},&\text{if $\mu=(f,n-2f-1)$},
\end{cases}
\end{align*}
and, let $N^\mu$ denote the $\mathcal{A}_{n-1}(x)$--submodule of $C^\lambda$ generated by $y^\lambda_\mu$.

In the next two lemmas, we assume that Lemma~\ref{stronger} is valid when applied to the algebra $\mathcal{A}_{n-1}(x)$ and show that the lemma is also true when applied to the algebra $\mathcal{A}_n(x)$ in the case that $\lambda$ is maximal among partitions of $n$. 
\begin{lemma}\label{prel:0}
Let $n=2f$ and $\lambda=(f,0)$. If $\mu=(f-1,1)$, then $\{y_\mu^\lambda v_\mathfrak{s}:\mathfrak{s}\in\mathfrak{T}_{n-1}(\mu)\}$ generates $C^\lambda$ as an $R$--module, and the $R$--module map $C^\lambda\to C^\mu$ determined by 
\begin{align*}
y^\lambda_\mu v_{\mathfrak{s}}\mapsto m_\mathfrak{t},&&\text{for $\mathfrak{t}\in\mathfrak{T}_{n}(\lambda)$ and $\mathfrak{s}=\mathfrak{t}|_{n-1}\in\mathfrak{T}_{n-1}(\mu)$,} 
\end{align*}
is an isomorphism of $\mathcal{A}_{n-1}(x)$--modules.
\end{lemma}
\begin{proof}
Let $b\in\mathcal{A}_{n-1}(x)$. Since $\check{\mathcal{A}}_{n-1}^\mu=0$, by Lemma~\ref{stronger}, which we apply inductively, there exist $a_\mathfrak{v}\in R$, for $\mathfrak{v}\in\mathfrak{T}_{n-1}(\mu)$, depending only on $b$, such that 
\begin{align*}
m_\mu b=\sum_{\mathfrak{v}\in\mathfrak{T}_{n-1}(\mu)} a_\mathfrak{v}m_\mu v_\mathfrak{v}.
\end{align*}
Since $m_\lambda= e_{2f-1}m_\mu$, we multiply both sides of the above expression by $e_{2f-1}$ on the left to obtain
\begin{align*}
m_\lambda b=e_{2f-1}m_\mu b=\sum_{\mathfrak{v}\in\mathfrak{T}_{n-1}(\mu)} a_\mathfrak{v}e_{2f-1}m_\mu v_\mathfrak{v}=\sum_{\mathfrak{u}\in\mathfrak{T}_{n}(\lambda )} a_\mathfrak{u}m_\lambda v_\mathfrak{u},
\end{align*}
where $a_\mathfrak{v}=a_\mathfrak{u}$ whenever $\mathfrak{u}|_{n-1}=\mathfrak{v}$. 

Now we show that the collection $\{y_\mu^\lambda v_\mathfrak{s}:\mathfrak{s}\in\mathfrak{T}_{n-1}(\mu)\}$ generates $C^\lambda$ as an $R$--module. If $\mathfrak{t}\in\mathfrak{T}_{n}(\lambda)$, and $\mathfrak{s}=\mathfrak{t}|_{n-1}$, then either 
\begin{enumerate}
\item[(i)] $\SHAPE(\mathfrak{s}|_{n-2})=(f-1,0)$, in which event $v_\mathfrak{s}\in\mathcal{A}_{n-2}(x)$, or
\item[(ii)] $\SHAPE(\mathfrak{s}|_{n-2})=(f-2,2)$, in which event $v_\mathfrak{t}=e_{n-2}v_\mathfrak{u}$, where $v_\mathfrak{u}\in\mathcal{A}_{n-2}(x)$.
\end{enumerate}
In the case (i), $m_\lambda v_\mathfrak{s}e_{n-1}=xm_\lambda v_\mathfrak{s}$, while, in the case (ii), $m_\lambda v_\mathfrak{t}e_{n-1}=m_\lambda e_{n-2}v_\mathfrak{u}e_{n-1}=m_\lambda v_\mathfrak{u}$ which, since $v_\mathfrak{u}\in\mathcal{A}_{n-1}(x)$, can be written as 
\begin{align*}
m_\lambda v_\mathfrak{u}=\sum_{\mathfrak{v}\in\mathfrak{T}_{n-1}(\mu)}a_\mathfrak{v}m_\lambda v_\mathfrak{v},
\end{align*}
where $a_\mathfrak{v}\in R$, for $\mathfrak{v}\in\mathfrak{T}_{n-1}(\mu)$. Thus, if $b\in\mathcal{A}_{n-1}(x)e_{n-1}\mathcal{A}_{n-1}(x)$, then $m_\lambda b$ can be expressed as an $R$--linear combination of terms from $\{y_\mu^\lambda v_\mathfrak{s}:\mathfrak{s}\in\mathfrak{T}_{n-1}(\mu)\}$. This completes the proof of the lemma.
\end{proof}

\begin{lemma}\label{prel:1}
Let $n=2f+1$ and $\lambda=(f,1)$. If $\mu^{(1)}=(f,0)$ and $\mu^{(2)}=(f-1,2)$, then 
\begin{align}\label{filter:0}
(0)=N^{\mu^{(0)}}\subseteq N^{\mu^{(1)}}\subseteq N^{\mu^{(2)}} = C^\lambda
\end{align}
is a filtration of the $\mathcal{A}_n(x)$--module $C^\lambda$ by $\mathcal{A}_{n-1}(x)$--modules. Moreover, if $\mu\in\{\mu^{(1)},\mu^{(2)}\}$, then $\{ y^\lambda_{\mu}v_\mathfrak{s}:\mathfrak{s}\in\mathfrak{T}_{n-1}(\mu)\}$ freely generates $N^\mu$ as an $R$--module, and the $R$--module homomorphism $N^{\mu^{(i)}}/N^{\mu^{(i-1)}}\mapsto C^\lambda $ determined by 
\begin{align}\label{prel:1:map}
y^\lambda_\mu v_{\mathfrak{s}}+N^{\mu^{(i-1)}}\mapsto m_\mathfrak{t},&&\text{for $\mathfrak{t}\in\mathfrak{T}_{n}(\lambda)$ and $\mathfrak{s}=\mathfrak{t}|_{n-1}\in\mathfrak{T}_{n-1}(\mu)$,} 
\end{align}
is an isomorphism of $\mathcal{A}_{n-1}(x)$--modules.
\end{lemma}
\begin{proof}
We have $y^\lambda_{\mu^{(1)}}=m_{\mathfrak{t}^\lambda}$ and $y^\lambda_{\mu^{(2)}}=m_{\mathfrak{t}^\lambda} e_{2f}$, so $y^\lambda_{\mu^{(2)}}e_{2f-1}=y^\lambda_{\mu^{(1)}}$, which shows that $N^{\mu^{(1)}}\subseteq N^{\mu^{(2)}}$ is an inclusion of $\mathcal{A}_{n-1}(x)$--modules. Furthermore, if $b\in\mathcal{A}_{n-1}(x)$ and $\mathfrak{u}\in\mathfrak{T}_{n-1}(\mu^{(1)})$, then, by Lemma~\ref{prel:0}, there exist $a_\mathfrak{v}\in R$, for $\mathfrak{v}\in\mathfrak{T}_{n-1}(\lambda)$, such that 
\begin{align*}
m_{\mu^{(1)}}v_\mathfrak{u} b=\sum_{\mathfrak{v}\in\mathfrak{T}_n(\mu^{(1)}) }a_\mathfrak{v}m_{\mu^{(1)}} v_\mathfrak{v};
\end{align*}
thus,
\begin{align*}
y^\lambda_{\mu^{(1)}}v_\mathfrak{u}b=\sum_{\mathfrak{v}\in\mathfrak{T}_{n-1}(\mu^{(1)})}a_\mathfrak{v}y^\lambda_{\mu^{(1)}} v_\mathfrak{v}.
\end{align*}

If $\mathfrak{u}\in\mathfrak{T}_{n-1}(\mu^{(2)})$, then $y_{\mu^{(2)}}^\lambda v_\mathfrak{u}b=e_{2f-1}e_{2f}m_{\mu^{(2)}}v_\mathfrak{u}b+\check{\mathcal{A}}_{n-1}^\lambda$, and, by Lemma~\ref{stronger} which we apply inductively,
\begin{align*}
e_{2f-1}e_{2f}m_{\mu^{(2)}}v_\mathfrak{u}b &=\sum_{\mathfrak{v}\in\mathfrak{T}_{n-1}(\mu^{(2)})}a_\mathfrak{v}e_{2f-1}e_{2f}m_{\mu^{(2)}}v_\mathfrak{v}+\sum_{\mathfrak{r},\mathfrak{t}\in\mathfrak{T}_{n-1}(\mu^{(1)})}a_{\mathfrak{rt}}e_{2f-1}e_{2f} v^*_\mathfrak{r}m_{\mu^{(1)}} v_\mathfrak{t},
\end{align*}
which shows that 
\begin{align*}
y_{\mu^{(2)}}^\lambda v_\mathfrak{u}b&=\sum_{{\mathfrak{v}\in\mathfrak{T}_{n-1}(\mu^{(2)})}}a_\mathfrak{v}y^{\lambda}_{\mu^{(2)}}v_\mathfrak{v}+\sum_{\mathfrak{r},\mathfrak{t}\in\mathfrak{T}_{n-1}(\mu^{(1)})}a_{\mathfrak{rt}}e_{2f-1}e_{2f} v^*_\mathfrak{r}m_{\mu^{(1)}} v_\mathfrak{t},
\end{align*}
where the sum is over $\mathfrak{r}\in\mathfrak{T}_{n-1}^{(f-1)}(\mu^{(1)})$. Since in fact $v_\mathfrak{r}=1$ whenever $\mathfrak{r}\in\mathfrak{T}_{n-1}^{(f-1)}(\mu^{(1)})$, and $e_{2f-1}e_{2f}m_{\mu^{(1)}}=m_\lambda=y^\lambda_{\mu^{(1)}}$, from the above expression, we obtain  
\begin{align*}
y_{\mu^{(2)}}^\lambda v_\mathfrak{u}b&=\sum_{{\mathfrak{v}\in\mathfrak{T}_{n-1}(\mu^{(2)})}}a_\mathfrak{v}y^{\lambda}_{\mu^{(2)}}v_\mathfrak{v}+\sum_{\mathfrak{r},\mathfrak{t}\in\mathfrak{T}_{n-1}(\mu^{(1)})}a_{\mathfrak{t}}m_{\mu^{(1)}} v_\mathfrak{t},
\end{align*}
which shows that the $R$--module map map $N^{\mu^{(2)}}/N^{\mu^{(1)}}\to C^{\mu^{(1)}}$ given by~\eqref{prel:1:map} is a homomorphism of $\mathcal{A}_{n-1}(x)$--modules. It remains to show that if $b\in\mathcal{A}_{n-1}(x)e_{n-1}\mathcal{A}_{n-1}(x)$, and $\mathfrak{t}\in\mathfrak{T}_{n}(\lambda)$, then $m_\lambda v_\mathfrak{t}b$ can be expressed as an $R$--linear combination of elements from $\{m_\lambda v_\mathfrak{s}:\mathfrak{s}\in\mathfrak{T}_n(\lambda)\}$.

Let $\mathfrak{t}\in\mathfrak{T}_{n}(\lambda)$, $\mathfrak{t}|_{n-1}\in\mathfrak{T}_{n-1}(\mu^{(1)})$ and $\mathfrak{u}=\mathfrak{t}|_{n-2}$. Then $\SHAPE(\mathfrak{u})=(f-1,1)$ and $v_\mathfrak{t}=v_\mathfrak{u}\in\mathcal{A}_{n-2}(x)$. Hence $v_\mathfrak{t}e_{2f}=e_{2f}v_\mathfrak{u}=v_\mathfrak{v}$, where $\mathfrak{v}\in\mathfrak{T}_n(\lambda)$ satisfies $\mathfrak{v}|_{n-1}\in\mathfrak{T}_{n-1}(\mu^{(2)})$; It follows that $m_\lambda v_\mathfrak{t}e_{n-1}=m_\lambda e_{2f}v_\mathfrak{t} =m_\lambda v_\mathfrak{v}$, which shows that in this instance, if $b\in\mathcal{A}_{n-1}(x)e_{n-1}\mathcal{A}_{n-1}(x)$, then $m_\lambda v_\mathfrak{t} b$ can be expressed as an $R$--linear combination of the required form.

If $\mathfrak{t}\in\mathfrak{T}_n(\lambda)$, $\mathfrak{t}|_{n-1}\in\mathfrak{T}_{n-1}(\mu^{(2)})$, and $\mathfrak{u}=\mathfrak{t}|_{n-2}$, then either 
\begin{enumerate}
\item[(i)] $\SHAPE(\mathfrak{u})=(f-1,1)$, in which event $v_\mathfrak{u}\in\mathcal{A}_{n-2}(x)$ and $v_\mathfrak{t}=e_{2f}v_\mathfrak{u}$, or
\item[(ii)] $\SHAPE(\mathfrak{u})=(f-2,3)$, in which event $v_\mathfrak{t}=e_{2f}e_{2f-2}e_{2f-1}v_\mathfrak{u}$, where $v_\mathfrak{u}\in\mathcal{A}_{n-2}(x)$.
\end{enumerate}
In the case (i) above, 
\begin{align*}
m_\lambda v_\mathfrak{t} e_{n-1}=m_\lambda e_{2f}v_\mathfrak{u}e_{n-1}=m_\lambda e_{2f}v_\mathfrak{u}e_{2f}=xm_\lambda e_{2f}v_\mathfrak{u}=xm_\lambda v_\mathfrak{t},
\end{align*}
while, in case (ii),
\begin{align*}
m_\lambda v_\mathfrak{t}e_{n-1}=m_\lambda e_{2f}e_{2f-2}e_{2f-1}v_\mathfrak{u}e_{n-1}=m_\lambda e_{2f}e_{2f-2}e_{2f-1}v_\mathfrak{u}e_{2f}=m_\lambda e_{2f}e_{2f-2}v_\mathfrak{u}.
\end{align*}
Since $e_{2f-2}v_\mathfrak{u}\in\mathcal{A}_{n-1}(x)$, the rightmost term in the above equalities can be expressed as a linear combination of elements from $\{m_\lambda v_\mathfrak{s}:\mathfrak{s}\in\mathfrak{T}_n(\lambda)\}$. This completes the proof of the lemma.
\end{proof}
The following corollary provides the base case in the induction used in the proof of Lemma~\ref{prel:a}.
\begin{corollary}\label{intermediate}
Let $f$ be a non--negative integer, and $n=2f+\delta$, where $\delta\in\{0,1\}$ and $\lambda=(f,\delta)$. If Lemma~\ref{stronger} holds for $\mathcal{A}_{n-1}(x)$, then the set $\{m_{\mathfrak{uv}}=v_\mathfrak{u}m_\lambda v_\mathfrak{v}:\mathfrak{u},\mathfrak{v}\in\mathfrak{T}_k(\lambda)\}$ freely generates $\mathcal{A}_n^\lambda$ as an $R$--module. Furthermore, if $u\in\mathfrak{T}_n(\lambda)$ and $b\in\mathcal{A}_n(x)$, then there exist 
$a_\mathfrak{v}\in R$, for $\mathfrak{v}\in\mathfrak{T}_{n-1}(\lambda)$, depending only on $\mathfrak{u}$ and $b$, such that 
\begin{align*}
m_\lambda v_\mathfrak{u}b=\sum_{\mathfrak{v}\in\mathfrak{T}_n(\lambda)}a_\mathfrak{v}m_\lambda v_\mathfrak{v}.
\end{align*}
\end{corollary}
In the next lemma, we take $\lambda=(f,n-2f)$ to be a partition with $n-2f>1$, and, using Corollary~\ref{intermediate}, assume that Lemma~\ref{stronger} holds for for $\mathcal{A}_{n-1}(x)$ and for $\mathcal{A}_n(x)$ in the case of partitions $\nu\rhd\lambda$. 
\begin{lemma}\label{prel:a}
Let $\lambda=(f,n-2f)$, where $n>n-2f>1$. If $\mu^{(1)}=(f,n-2f-1)$ and $\mu^{(2)}=(f-1,n-2f+1)$. Then 
\begin{align}\label{filter:1}
(0)=N^{\mu^{(0)}}\subseteq N^{\mu^{(1)}}\subseteq N^{\mu^{(2)}} = C^\lambda
\end{align}
is a filtration of the $\mathcal{A}_n(x)$--module $C^\lambda$ by $\mathcal{A}_{n-1}(x)$--modules. Moreover, if $\mu\in\{\mu^{(1)},\mu^{(2)}\}$, then $\{ y^\lambda_{\mu}v_\mathfrak{s}:\mathfrak{s}\in\mathfrak{T}_{n-1}(\mu)\}$ freely generates $N^\mu$ as an $R$--module, and the $R$--module homomorphism $N^{\mu^{(i)}}/N^{\mu^{(i-1)}}\mapsto C^\lambda $ determined by 
\begin{align}\label{prel:iso}
y^\lambda_\mu v_{\mathfrak{s}}+N^{\mu^{(i-1)}}\mapsto m_\mathfrak{t},&&\text{for $\mathfrak{t}\in\mathfrak{T}_{n}(\lambda)$ and $\mathfrak{s}=\mathfrak{t}|_{n-1}\in\mathfrak{T}_{n-1}(\mu)$,} 
\end{align}
is an isomorphism of $\mathcal{A}_{n-1}(x)$--modules. 
\end{lemma}
\begin{proof}
First, if $\nu=(k,n-2k-1)\rhd\mu^{(1)}$, and $b\in\mathcal{A}_{n-1}^\nu$, then $C^\lambda b=0$, since $\mathcal{A}^\nu_{n-1}\subset\check{\mathcal{A}}_n^\lambda$
.

Next, observe that, if we write $\mu=(2f,n-2f-1)$, then, consistent with the inclusion of algebras in~\eqref{tower}, the $\mathcal{A}_{n-1}(x)$--module $N^\mu$ is isomorphic to the $\mathcal{A}_{n-1}(x)$--module $C^\mu$. Thus, by induction,  $\{ y^\lambda_{\mu}v_\mathfrak{s}:\mathfrak{s}\in\mathfrak{T}_{n-1}(\mu)\}$ freely generates $N^\mu$ as an $R$--module.

If $\mu=(f-1,n-2f+1)$, then $y^\lambda_\mu w_{n-1,2f-1}=m_{\mathfrak{t}^\lambda}$, showing that~\eqref{filter:1} is an inclusion of $\mathcal{A}_{n-1}(x)$--modules. 

Now, let $\mu=(f-1,n-2f-1)$, suppose that $\mathfrak{t}\in\mathfrak{T}_{n-1}(\mu)$, and consider the action of an element $b\in\mathcal{A}_{n-1}(x)$ in the expression
\begin{align}\label{act:2}
m_\lambda w_{f,n}v_\mathfrak{t}b&=e_{2f-1}w_{2f,n}m_\mu v_\mathfrak{t}b.
\end{align}
By Lemma~\ref{stronger}, which we apply inductively, there exist $a_{\mathfrak{v}}\in R$, for $\mathfrak{v}\in\mathfrak{T}_{n}(\lambda)$, which depend only on $\mathfrak{t}$, such that
\begin{align}\label{expr:a}
m_\mu v_{\mathfrak{t}} b =\sum_{\mathfrak{v}\in\mathfrak{T}_{n}(\mu)} a_\mathfrak{v}m_\mu v_\mathfrak{v}+\sum_{\substack{\nu\rhd\mu\\ \mathfrak{s},\mathfrak{u}\in\mathfrak{T}_{n-1}(\nu)}}a_{\mathfrak{su}}v^*_\mathfrak{s}m_\nu v_{\mathfrak{u}},
\end{align}
where the latter sum is over partitions $\nu=(k,n-2k-1)$, for $k=f,f+1,\dots$, and $\mathfrak{s}\in\mathfrak{T}_{n-1}^{(f-1)}(\nu)$. Substituted into~\eqref{act:2}, the expression~\eqref{expr:a} gives
\begin{align}
m_\lambda w_{2f,n}v_\mathfrak{t}b&=\sum_{\mathfrak{v}\in\mathfrak{T}_{n}(\mu)} a_\mathfrak{v}e_{2f-1}w_{2f,n}m_\mu v_\mathfrak{v}+\sum_{\substack{\nu\rhd\mu\\ \mathfrak{s},\mathfrak{u}\in\mathfrak{T}_n(\nu)}}a_{\mathfrak{su}}e_{2f-1}w_{2f,n}v^*_\mathfrak{s}m_\nu v_{\mathfrak{u}}\label{indstep:1}\\
&=\sum_{\mathfrak{v}\in\mathfrak{T}_{n}(\mu)} a_\mathfrak{v}m_\lambda w_{2f,n} v_\mathfrak{v}+\sum_{\substack{\nu\rhd\mu\\ \mathfrak{s},\mathfrak{u}\in\mathfrak{T}_n(\nu)}}a_{\mathfrak{su}}e_{2f-1}w_{2f,n}v^*_\mathfrak{s}m_\nu v_{\mathfrak{u}}\notag\\
&=\sum_{\substack{\mathfrak{r}\in\mathfrak{T}_{n}(\lambda),\\
\mathfrak{r}|_{n-1}\in\mathfrak{T}_{n-1}(\mu)}} a_\mathfrak{r}m_\lambda v_\mathfrak{r}+\sum_{\substack{\nu\rhd\mu\\\mathfrak{s},\mathfrak{u}\in\mathfrak{T}_n(\nu)}}a_{\mathfrak{su}}e_{2f-1}w_{2f,n}v^*_\mathfrak{s}m_\nu v_{\mathfrak{u}},\notag
\end{align}
where, in the above expression, $a_\mathfrak{r}\in R$ are defined, for $\mathfrak{r}\in\mathfrak{T}_n(\lambda)$, by the condition that $a_\mathfrak{v}=a_\mathfrak{r}$ whenever $\mathfrak{r}|_{n-1}=\mathfrak{v}\in\mathfrak{T}_{n-1}(\mu)$. 

Now, using Lemma~\ref{techs:1}, we turn our consideration to the summands $e_{2f-1}w_{2f,n}v^*_\mathfrak{s}m_\nu v_{\mathfrak{u}}$ appearing in~\eqref{indstep:1}. Let $\tau=(f+1,n-2f-2)$; if $\nu=(f,n-2f-1)$, then either 
\begin{align}\label{expar:1}
e_{2f-1}w_{2f,n}v^*_\mathfrak{s}m_\nu v_{\mathfrak{u}}=m_\lambda v_{\mathfrak{v}'},&&\text{where $\mathfrak{v}'\in\mathfrak{T}_{n-1}(\lambda)$ and $\SHAPE(\mathfrak{v}'|_{n-1})=\mu^{(1)}$,}
\end{align}
or, there exists $\mathfrak{t}'\in\mathfrak{T}_n^{(f)}(\tau)$, such that 
\begin{align}\label{expar:2}
e_{2f-1}w_{2f,n}v_\mathfrak{s}^* m_\nu v_\mathfrak{u}= v_{\mathfrak{t}'}^* m_\tau v_{\mathfrak{v}'},&&\text{where $\mathfrak{v'}\in\mathfrak{T}_n(\tau)$ and $\mathfrak{v}'|_{n-1}=\mathfrak{u}$.}
\end{align}
Now suppose that $\nu=(k,n-2k-1)$, where $k=f+1,f+2,\dots$, let $\mathfrak{s}\in\mathfrak{T}_{n-1}^{(f-1)}(\nu)$, and consider the product $e_{2f-1}w_{2f,n}v_{\mathfrak{s}}^*m_\nu$. We may write 
\begin{align*}
v_\mathfrak{s}^*=w_{j_0,2i}w_{j_1,2i+2}\cdots w_{r,2k},&&\text{where $f\le i\le k$ and $2i\le j_0<j_1<\cdots \le r\le n-1$.}
\end{align*}
If $v_\mathfrak{s}^*=1$, then 
\begin{align*}
e_{2f-1}w_{2f,n}v_{\mathfrak{s}}^*m_\nu=e_{2f-1}w_{2f,n}m_\nu=e_{2f-1}w_{2f,n}e_1e_3\cdots e_{2k-1}=w_{2k+2,n}e_1e_3\cdots e_{2k+1};
\end{align*}
otherwise, if $f<i$, so that $v_\mathfrak{s}^*\in\langle e_{2f+2},\dots,e_{n-2}\rangle$, then 
\begin{align*}
e_{2f-1}w_{2f,n}v_\mathfrak{s}^*m_\nu&=e_{2f-1}w_{2f,n}v_\mathfrak{s}^*e_1e_3\cdots e_{2k-1}\\
&=e_{2f+1}w_{2f+2,n}v_\mathfrak{s}^*e_1e_3\cdots e_{2k-1}\\
&=e_{2f+1}w_{2f+2,n}v_\mathfrak{s}^*m_\nu.
\end{align*}
Thus we suppose that $v_\mathfrak{s}^*=w_{j_0,2f}w_{j_1,2f+2}\cdots w_{r,2k}$ where $2f<j_0<j_1<\cdots< r$, in which event,
\begin{align}\label{techs:cas}
w_{2f,n}w_{j_0,2f}=
\begin{cases}
e_{2f},&\text{if $j_0=n-1$;}\\
e_{2f}e_{j_0+1}e_{j_0+2}\cdots e_{n-1},&\text{if $2f<j_0<n-1$.}
\end{cases}
\end{align}
Let $v_{\mathfrak{s}'}^*=w_{j_1,2f+2}\cdots w_{r,2k}$; in the first case in~\eqref{techs:cas}, using $e_{2f-1}e_{2f}e_{2f-1}=e_{2f-1}$, 
\begin{align*}
e_{2f}w_{2f,n}v_\mathfrak{s}^* m_\nu=e_{2f-1}w_{2f,n}w_{n-1,2f}v_{\mathfrak{s}'}^*m_\nu=e_{2f-1}e_{2f}v_{\mathfrak{s}'}^*m_\nu=v_{\mathfrak{s}'}^*m_\nu,
\end{align*}
and in the second, 
\begin{align*}
e_{2f-1}w_{2f,n}v_\mathfrak{s}^* m_\nu&=e_{2f-1}e_{2f}e_{j_0+1}e_{j_0+2}\cdots e_{n-1}v_{\mathfrak{s}'}^*m_\nu\\
&=e_{j_0+1}e_{j_0+2}\cdots e_{n-1}v_{\mathfrak{s}'}^*m_\nu.
\end{align*}
Let $\nu'=(k,n-2k)$, so that $m_\nu\mapsto m_{\nu'}$ under the inclusion $\mathcal{A}_{n-1}(x)\hookrightarrow\mathcal{A}_{n}(x)$, and let $b'=e_{2f-1}w_{2f,n}v_\mathfrak{s}$ take a value in  $\{v^*_{\mathfrak{s}'},e_{j_0+1}e_{j_0+2}\cdots e_{n-1}v_{\mathfrak{s}'}^*\}$, depending on the case in~\eqref{techs:cas}. Since $b'\in \langle e_{i}:2f<i<n\rangle$, and $\nu'\rhd\lambda$, there exist $a_{\mathfrak{s}''}\in R$, for $\mathfrak{s}''\in \mathfrak{T}_n(\nu')$, such that 
\begin{align}\label{expar:a}
e_{2f-1}w_{2f,n}v_\mathfrak{s}m_\nu=b'm_{\nu'}=\sum_{\mathfrak{s}''\in\mathfrak{T}_n^{(f)}(\nu')} v_{\mathfrak{s}''}^* m_{\nu'} + \sum_{\substack{\tau'\rhd\nu'\\ \mathfrak{r}',\mathfrak{t}'\in \mathfrak{T}_n(\tau')}} a_{\mathfrak{r}'\mathfrak{t}'} v_{\mathfrak{r}'}^*m_{\tau'} v_{\mathfrak{t}'},
\end{align}
where $\mathfrak{r}'\in\mathfrak{T}_n^{(f)}(\tau')$, $\mathfrak{t}'\in \mathfrak{T}_n^{(k)}(\tau')$, and $a_{\mathfrak{r}'\mathfrak{t}'}\in R$, for $\tau'\rhd\nu'$. If $\mathfrak{u}\in\mathfrak{T}_{n-1}(\nu)$, then $v_\mathfrak{u}=v_{\mathfrak{u}'}$, where $\mathfrak{u}'|_{n-1}=\mathfrak{u}\in\mathfrak{T}_{n-1}(\nu)$ and $\mathfrak{u}\in\mathfrak{T}_n(\nu')$, we multiply both sides of~\eqref{expar:a} by $v_\mathfrak{u}=v_{\mathfrak{u}'}$ to obtain
\begin{align}\label{expar:b}
e_{2f-1}w_{2f,n}v_\mathfrak{s}m_\nu v_\mathfrak{u}=\sum_{\mathfrak{s}''\in\mathfrak{T}_n^{(f)}(\nu')} v_{\mathfrak{s}''}^* m_{\nu'}v_{\mathfrak{u}'} + \sum_{\substack{\tau'\rhd\nu'\\ \mathfrak{r}'',\mathfrak{t}''\in \mathfrak{T}_n(\tau')}} a_{\mathfrak{r}''\mathfrak{t}''} v_{\mathfrak{r}''}^*m_{\tau'} v_{\mathfrak{t}''},
\end{align}
where $\mathfrak{r}''\in\mathfrak{T}_n^{(f)}(\tau')$, $\mathfrak{t}''\in \mathfrak{T}_n(\tau')$, and $a_{\mathfrak{r}''\mathfrak{t}''}\in R$, for $\tau'\rhd\nu'$. 

Combining~\eqref{indstep:1} with~\eqref{expar:1},~\eqref{expar:2} and~\eqref{expar:b}, we have shown that if $\mathfrak{t}\in\mathfrak{T}_n(\lambda)$, $\SHAPE(\mathfrak{t}|_{n-1})=\mu^{(2)}$ and $b\in\mathcal{A}_{n-1}(x)$, then there exist $a_\mathfrak{r},a_{\mathfrak{v}'}\in R$, for $\mathfrak{r},\mathfrak{v}'\in\mathfrak{T}_n(\lambda)$, where $\SHAPE(\mathfrak{r}|_{n-1})=\mu^{(2)}$ and $\SHAPE(\mathfrak{v}'|_{n-1})=\mu^{(1)}$, satisfying 
\begin{align}\label{expar:3}
m_\lambda v_{\mathfrak{t}}b&=\sum_{\substack{\mathfrak{r}\in\mathfrak{T}_n(\lambda)\\\SHAPE(\mathfrak{r}|_{n-1})=\mu^{(2)}}}a_\mathfrak{r}m_\lambda v_\mathfrak{r}+\sum_{\substack{\mathfrak{v}'\in\mathfrak{T}_n(\lambda)\\\SHAPE(\mathfrak{v}'|_{n-1})=\mu^{(1)}}}a_{\mathfrak{v}'}m_\lambda v_{\mathfrak{v}'}+\sum_{\substack{\gamma\rhd\lambda\\ \mathfrak{s},\mathfrak{u}\in\mathfrak{T}_n(\gamma)}}a_{\mathfrak{su}}v_\mathfrak{s}^*m_\gamma v_\mathfrak{u},
\end{align}
where the sum is over $\mathfrak{s}\in\mathfrak{T}_n^{(f)}(\gamma)$, $\mathfrak{u}\in\mathfrak{T}_n(\gamma)$, and $a_{\mathfrak{su}}\in R$, for $\gamma\rhd\lambda$. The manner in which the $a_\mathfrak{r}\in R$, for $\mathfrak{r}\in\mathfrak{T}_n(\lambda)$ satisfying $\SHAPE(\mathfrak{r}|_{n-1})=\mu^{(2)}$, are derived in~\eqref{indstep:1} from the action of $\mathcal{A}_{n-1}(x)$ on $C^{\mu^{(2)}}$ shows that the map~\eqref{prel:iso} is a homomorphism of $\mathcal{A}_{n-1}(x)$--modules.

It remains to demonstrate that $N^{\mu^{(2)}}=C^\lambda$. To this purpose, we show that if $\mathfrak{t}\in\mathfrak{T}_{n}(\lambda)$ and $b\in\mathcal{A}_{n-1}(x)e_{n-1}\mathcal{A}_{n-1}(x)$, then $m_\lambda v_\mathfrak{t}b$ can be expressed as a sum of the form~\eqref{expar:3}. Firstly, we suppose that $\mathfrak{t}\in\mathfrak{T}_{n}(\lambda)$ and let $\mathfrak{u}=\mathfrak{t}_{n-1}$ satisfy $\SHAPE(\mathfrak{u})=\mu^{(1)}$. In this case, $v_\mathfrak{t}=v_\mathfrak{u}\in\mathcal{A}_{n-1}(x)$. If $v_\mathfrak{u}\in\mathcal{A}_{n-2}(x)$, then 
\begin{align}\label{expar:4}
m_\lambda v_\mathfrak{t}e_{n-1}=m_\lambda e_{n-1} v_\mathfrak{t} =w_{n,2f+2}m_\nu w_{2f+2,n}v_\mathfrak{t}, &&\text{where $\nu=(f+1,n-2f-2)$.}
\end{align}
By what we have already shown, the term appearing on the right hand side of~\eqref{expar:4} can be written as a sum of the form~\eqref{expar:3}. Otherwise, if $v_\mathfrak{u}=w_{2f,n-1}v_\mathfrak{v}$, where $\SHAPE(\mathfrak{v})=(f-1,n-2f)$ and $v_\mathfrak{v}\in\mathcal{A}_{n-2}(x)$, then 
\begin{align}\label{expar:5}
m_\lambda v_\mathfrak{t}e_{n-1}=m_\lambda w_{2f,n-1}v_\mathfrak{v}e_{n-1}=m_\lambda w_{2f,n-1}e_{n-1}v_\mathfrak{v}=m_\lambda w_{2f,n}v_\mathfrak{v}=m_\lambda v_\mathfrak{s}
\end{align}
where $\mathfrak{s}\in\mathfrak{T}_n(\lambda)$ is defined by $\mathfrak{s}|_{n-2}=\mathfrak{v}$ and $\SHAPE(\mathfrak{s}|_{n-1})=(f-1,n-2f+1)$. Now suppose that $v_\mathfrak{t}=w_{2f,n}v_\mathfrak{u}$, where $\SHAPE(\mathfrak{u})=(f-1,n-2f+1)$. If $v_\mathfrak{u}\in\mathcal{A}_{n-2}(x)$, then $m_\lambda v_\mathfrak{t} e_{n-1}=xm_\lambda v_\mathfrak{t}$; otherwise $v_\mathfrak{u}=w_{2f-2,n-1}v_\mathfrak{v}$, where $\SHAPE(\mathfrak{v})=(f-2,n-2f+2)$ and $v_\mathfrak{v}\in\mathcal{A}_{n-2}(x)$. Thus,
\begin{align*}
m_\lambda v_\mathfrak{t}e_{n-1}=m_\lambda w_{2f,n}w_{2f-2,n-1}v_\mathfrak{v}e_{n-1}=m_\lambda w_{2f,n}w_{2f-2,n}v_\mathfrak{v}=m_\lambda w_{2f,n}w_{2f-2,n-2}v_\mathfrak{v},
\end{align*}
which is a term that we have already shown can be expressed as a sum of the form~\eqref{expar:3}. This completes the proof of the lemma.
\end{proof}
\begin{proof}[Proof of Lemma~\ref{stronger}]
Firstly, if $b\in\mathcal{A}_{n-1}(x)$, then~\eqref{btsu:2} holds by virtue of the calculations preceding~\eqref{expar:3} in the proof of Lemma~\ref{intermediate} and, if $b\in\mathcal{A}_{n-1}(x)e_{n-1}\mathcal{A}_{n-1}(x)$, then the proof of the fact that $N^{\mu^{(2)}}=C^\lambda$ in the proof of Lemma~\ref{intermediate} shows that~\eqref{btsu:2} holds.
\end{proof}
\begin{example}\label{ex:mbasis}
If $n=6$ and $\lambda=(2,2)$, then the elements $m_\lambda v_\mathfrak{s}$, for $\mathfrak{s}\in\mathfrak{T}_n(\lambda)$, are given in terms of the diagram presentation for $\mathcal{A}_n(x)$ are as follows:
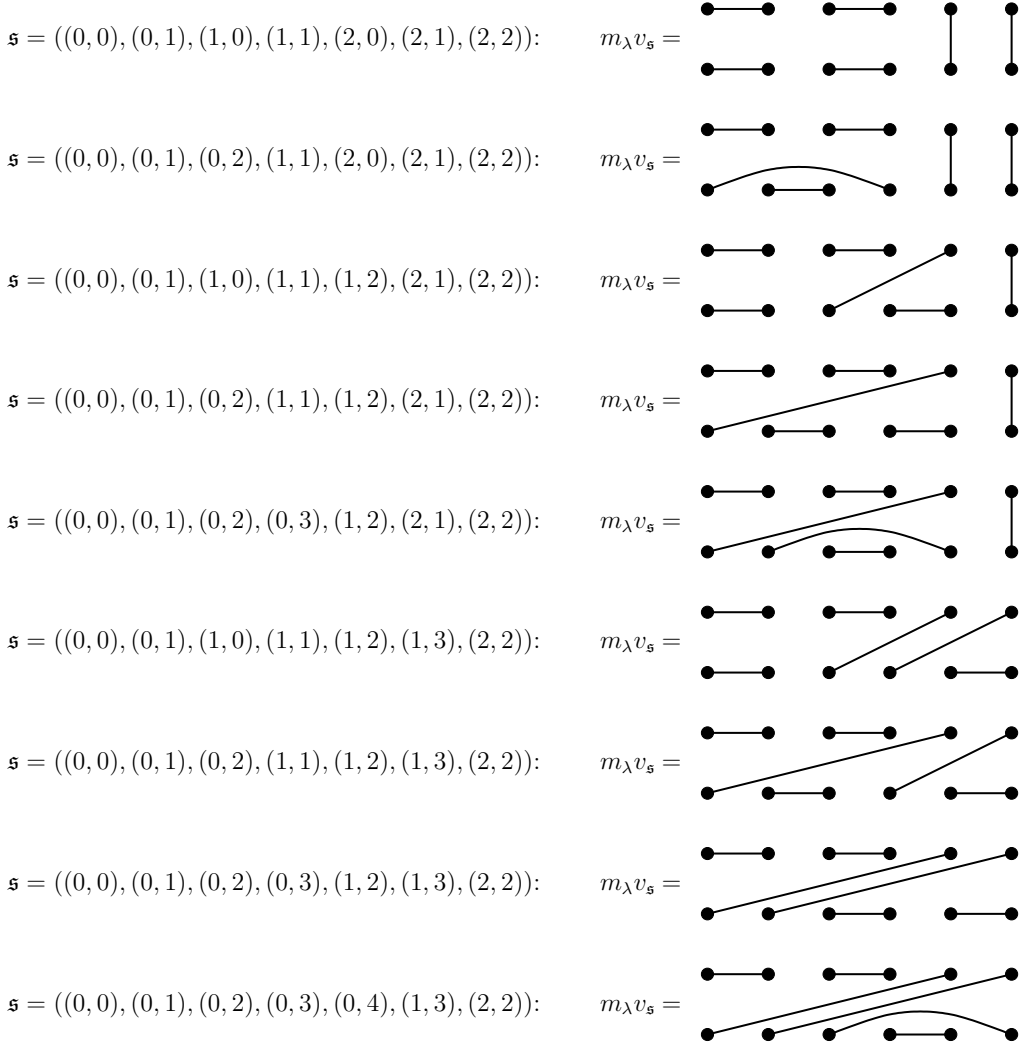
\begin{figure}[h]
\begin{center}
\scalebox{0.8}{
\input{fig1.pstex_t}}
\end{center}
%\caption{<+caption text+>}
%\label{fig:<+label+>}
\end{figure}
\end{example}

\section{Jucys--Murphy Elements for Temperley--Lieb Algebras}\label{jmelements}
In~\cite{hmram:2007}, T.~Halverson, M.~Mazzocco and A.~Ram have defined a family of commuting operators in the affine Temperley--Lieb algebras. The operators of~\cite{hmram:2007} are analogues to the Jucys--Murphy elements from the representation theory of the symmetric group. For the purposes of these notes, it will be useful to define in $\mathcal{A}_n(x)$ a sequence $(T_i:i=0,1,\dots)$ by $T_0=0$, $T_1=0$, $T_2=e_1$, and 
\begin{align*}
T_{i+1}=-e_iT_i-T_ie_i+e_ie_{i-1}T_ie_i-z_{i-1}-e_iz_{i-2}, &&\text{for $i=2,3,\dots$}
\end{align*}
where, $z_1=0$, and 
\begin{align*}
z_i=\displaystyle{\sum_{k=0}^ix^kT_{i-k}},&&\text{for $i=2,3,\dots$.}
\end{align*}
\begin{lemma}
For $i=2,3,\dots,$ the following statements hold:
\begin{enumerate}
\item $e_{i+1}e_{i}T_{i}e_{i+1}=e_{i+1}T_{i}e_{i}e_{i+1}$;\label{comm:1}
\item $e_ie_{i-1}T_ie_i=e_iT_ie_{i-1}e_i$;\label{comm:2}
\item $z_i^*=z_i$ and $T_i^*=T_i$;\label{comm:3}
\item $e_{i-1}T_{i+1}=T_{i+1}e_{i-1}$;\label{comm:4}
\item $e_i(xT_{i}+T_{i+1})=(xT_{i}+T_{i+1})e_i$;\label{comm:6}
\item $T_i$ commutes with $\mathcal{A}_{i-1}(x)$;\label{comm:5}
\item the element $z_i=\sum_{k=0}^{i}x^{k}T_{i-k}$ is central in $\mathcal{A}_i(x)$\label{comm:7}.
\end{enumerate}
\end{lemma}
\begin{proof}
The item~(\ref{comm:1}) follows from the relation $e_{i+1}e_ie_{i+1}=e_{i+1}$ and fact that $e_{i+1}$ commutes with $T_i$. Turning to the statement~(\ref{comm:2}), which is true when $i=2$, we proceed by induction. 
Since
\begin{align*}
T_{i}=-e_{i-1}T_{i-1}-T_{i-1}e_{i-1}+e_{i-1}e_{i-2}T_{i-1}e_{i-1}-z_{i-2}-z_{i-3}e_{i-1},
\end{align*}
applying~(\ref{comm:1}) yields 
\begin{align*}
e_{i}e_{i-1}T_{i}e_i&=-xe_{i}e_{i-1}T_{i-1}e_i-e_{i}e_{i-1}T_{i-1}e_{i-1}e_i\\
&\qquad\quad+xe_ie_{i-1}e_{i-2}T_{i-1}e_{i-1}e_i-e_ie_{i-1}z_{i-2}e_i-xe_ie_{i-1}z_{i-3}e_i\\
&=xe_{i}T_{i-1}-e_{i}e_{i-1}T_{i-1}e_{i-1}e_i\\
&\qquad\quad+xe_ie_{i-1}T_{i-1}e_{i-2}e_{i-1}e_i-e_ie_{i-1}e_iz_{i-2}-xe_ie_{i-1}e_iz_{i-3}\\
&=xe_{i}T_{i-1}-e_{i}e_{i-1}T_{i-1}e_{i-1}e_i\\
&\qquad\quad+xe_ie_{i-1}e_{i-2}T_{i-1}e_{i-1}e_i-e_iz_{i-2}-xe_iz_{i-3}\\
&=e_{i}T_{i}e_{i-1}e_i,
\end{align*}
as required. The statement~(\ref{comm:3}) follows from~(\ref{comm:2}), while 
\begin{align}
e_{i-1}T_{i+1}&=-e_{i-1}e_iT_i-e_{i-1}T_ie_i+e_{i-1}e_ie_{i-1}T_ie_i-e_{i-1}z_{i-1}-e_{i-1}z_{i-2}e_{i}\label{al:com}\\
&=-e_{i-1}e_iT_i-e_{i-1}T_ie_i+e_{i-1}T_ie_i-e_{i-1}z_{i-1}-e_{i-1}z_{i-2}e_{i}\notag\\
&=-e_{i-1}e_iT_i-e_{i-1}T_{i-1}-e_{i-1}z_{i-1}-e_{i-1}z_{i-2}e_i\notag.
\end{align}
Now,
\begin{align*}
e_{i-1}e_iT_i+e_{i-1}T_{i-1}&=-e_{i-1}e_ie_{i-1}T_{i-1}-e_{i-1}e_iT_{i-1}e_{i-1}\\
&\qquad+e_{i-1}e_ie_{i-1}e_{i-2}T_{i-1}e_{i-1}-e_{i-1}e_iz_{i-2}-e_{i-1}e_iz_{i-3}e_{i-1}+e_{i-1}T_{i-1}\\
&=-e_{i-1}T_{i-1}-e_{i-1}e_iT_{i-1}e_{i-1}\\
&\qquad+e_{i-1}e_{i-2}T_{i-1}e_{i-1}-e_{i-1}e_iz_{i-2}-e_{i-1}z_{i-3}+e_{i-1}T_{i-1}\\
&=-e_{i-1}e_iT_{i-1}e_{i-1}+e_{i-1}e_{i-2}T_{i-1}e_{i-1}-e_{i-1}e_iz_{i-2}-e_{i-1}z_{i-3},
\end{align*}
which, substituted into~\eqref{al:com}, yields
\begin{align*}
e_{i-1}T_{i+1}&=e_{i-1}e_iT_{i-1}e_{i-1}-e_{i-1}e_{i-2}T_{i-1}e_{i-1}+e_{i-1}z_{i-2}e_i-e_{i-1}z_{i-2}e_{i}+e_{i-1}z_{i-3}\\
&=e_{i-1}e_iT_{i-1}e_{i-1}-e_{i-1}e_{i-2}T_{i-1}e_{i-1}+e_{i-1}z_{i-3}\\
&=e_{i-1}T_{i-1}e_ie_{i-1}-e_{i-1}T_{i-1}e_{i-2}e_{i-1}+e_{i-1}z_{i-3}\\
&=T_{i+1}e_{i-1}.
\end{align*}
To see~(\ref{comm:6}), we have 
\begin{align*}
e_i(xT_i+T_{i+1})&=xe_iT_i-xe_iT_i-e_iT_ie_i+xe_ie_{i-1}T_ie_i-e_iz_{i-1}-xz_{i-2}e_i\\
&=-e_iT_ie_i+xe_ie_{i-1}T_ie_i-e_iz_{i-1}-xz_{i-2}e_i\\
&=-e_iT_ie_i+xe_iT_ie_{i-1}e_i-z_{i-1}e_i-xz_{i-2}e_i\\
&=(xT_i+T_{i+1})e_i.
\end{align*}
The proof of~(\ref{comm:5}) and~(\ref{comm:7}) is a joint induction. We assume that $e_k$ commutes with $T_i$ whenever $k=1,\dots,i-2$, and that $z_k$ is central in $\mathcal{A}_k(x)$ whenever $k=1,2,\dots,i$. Since it is already known that $T_{i+1}$ commutes with $e_{i-1}$, we first show that $T_{i+1}$ commutes with $e_{i-2}$. By item~(\ref{comm:4}),
\begin{align}
T_{i+1}e_{i-2}&=-e_{i}T_ie_{i-2}-T_ie_ie_{i-2}+e_{i}T_ie_{i-1}e_ie_{i-2}-z_{i-1}e_{i-2}-e_iz_{i-2}e_{i-2}\label{comm:e}\\
&=-e_{i-2}e_{i}T_i-e_{i-2}T_ie_i+e_{i}T_ie_{i-1}e_ie_{i-2}-e_{i-2}z_{i-1}-e_iz_{i-2}e_{i-2}.\notag
\end{align}
Hence we must show that $e_{i}T_{i}e_{i-1}e_i-e_iz_{i-2}$ commutes with $e_{i-2}$; to keep the indices within nice bounds, we demonstrate that $e_{k+1}T_{k+1}e_{k}e_{k+1}-e_{k+1}z_{k-1}$ commutes with $e_{k-1}$:
\begin{align*}
&e_{k+1}T_{k+1}e_ke_{k+1}e_{k-1}-e_{k+1}z_{k-1}e_{k-1}=-e_{k+1}e_kT_ke_ke_{k+1}e_{k-1}-xe_{k+1}T_ke_ke_{k+1}e_{k-1}\\
&\qquad+xe_{k+1}e_kT_ke_{k-1}e_ke_{k+1}e_{k-1}-e_{k+1}z_{k-1}e_ke_{k+1}e_{k-1}\\
&\qquad-xe_{k+1}e_kz_{k-2}e_{k+1}e_{k-1}-e_{k+1}z_{k-1}e_{k-1}\\
&=-e_{k+1}e_kT_ke_ke_{k+1}e_{k-1}-xe_{k+1}T_ke_{k-1}+xe_{k+1}T_ke_{k-1}-e_{k+1}z_{k-1}e_{k-1}\\
&\qquad-xe_{k+1}z_{k-2}e_{k-1}-e_{k+1}z_{k-1}e_{k-1}\\
&=-e_{k+1}e_kT_ke_ke_{k+1}e_{k-1}-xe_{k+1}e_{k-1}z_{k-2}-2e_{k+1}z_{k-1}e_{k-1}.
\end{align*}
Since $(e_{k+1}e_{k-1}z_{k-2})^*=e_{k+1}e_{k-1}z_{k-2}$, we consider
\begin{multline*}
e_{k+1}e_kT_ke_ke_{k+1}e_{k-1}+2e_{k+1}z_{k-1}e_{k-1}=-e_{k+1}e_ke_{k-1}T_{k-1}e_ke_{k+1}e_{k-1}\\
-e_{k+1}e_kT_{k-1}e_{k-1}e_ke_{k+1}e_{k-1}+e_{k+1}e_ke_{k-1}T_{k-1}e_{k-2}e_{k-1}e_ke_{k+1}e_{k-1}-e_{k+1}e_kz_{k-2}e_{k+1}e_{k-1}\\-e_{k+1}e_ke_{k-1}z_{k-3}e_ke_{k+1}e_{k-1}+2e_{k+1}z_{k-1}e_{k-1}\\
=-e_{k+1}T_{k-1}e_{k-1}-e_{k+1}T_{k-1}e_{k-1}+e_{k+1}e_{k-1}T_{k-1}e_{k-2}e_{k-1}-xe_{k+1}z_{k-2}e_{k-1}-e_{k+1}e_{k-1}z_{k-3}\\+2e_{k+1}z_{k-1}e_{k-1}\\
=e_{k+1}e_{k-1}T_{k-1}e_{k-2}e_{k-1}-xe_{k+1}z_{k-2}e_{k-1}-2e_{k+1}z_{k-2}e_{k-1}\\
=e_{k+1}e_{k-1}e_{k-2}T_{k-1}e_{k-1}-xe_{k-1}e_{k+1}z_{k-2}-2e_{k-1}e_{k+1}z_{k-2},
\end{multline*}
which shows that $e_{k+1}e_kT_ke_ke_{k+1}e_{k-1}+2e_{k+1}z_{k-1}e_{k-1}$ is fixed by $*:\mathcal{A}_{k+2}(z)\to\mathcal{A}_{k+2}(z)$, and therefore that $e_{i}T_{i}e_{i-1}e_{i}-e_{i}z_{i-2}$ commutes with $e_{i-2}$. Hence, by~\eqref{comm:e}, $T_{i+1}$ commutes with $e_{i-2}$. Now, if $k=1,\dots,i-3$, then by induction,
\begin{align*}
T_{i+1}e_k&=-e_iT_ie_k-T_ie_ie_k+e_ie_{i-1}T_ie_ie_k-z_{i}e_k-e_iz_{i-2}e_k\\
&=-e_ke_iT_i-e_kT_ie_i+e_ke_ie_{i-1}T_ie_i-e_kz_{i}-e_ke_iz_{i-2}=e_kT_{i+1}.
\end{align*}
The item~(\ref{comm:6}) shows that $z_{i+1}$ commutes with $e_1,\dots,e_i$, which completes the induction and the proof of the lemma.
\end{proof}
Since the elements of $(T_i:i=0,1,\dots)$ commute, we are justified in referring to the $T_i$ as ``Jucys--Murphy elements'' for the Temperley--Lieb algebras.

Define a sequence $(p_i\in R:i=0,1,\dots)$ by $p_0=0$, $p_1=1$ and $p_{i+1}=xp_i-p_{i-1}$ for $i=1,2,\dots$. For $k=1,2,\dots$, introduce a sequence $(T_i^{(k)}:i=0,1,\dots)$ by $T_0^{(k)}=0$, $T_1^{(k)}=0$, $T_2^{(k)}=e_{k}$, and, for $i=2,3,\dots$,
\begin{align*}
T_{i+1}^{(k)}=-e_{k+i-1}T_i^{(k)}-T_i^{(k)}e_{k+i-1}+e_{k+i-1}e_{k+i-2}T_i^{(k)}e_{k+i-1}-z_{i-1}^{(k)}-z_{i-2}^{(k)}e_{k+i-1},
\end{align*}
where  $z_1^{(k)}=0$, and 
\begin{align*}
z_i^{(k)}=\sum_{j=0}^ix^jT_{i-j}^{(k)},&&\text{for $i=2,3,\dots$,}
\end{align*}
so that the $T_{i}^{(k)}$, for $i=0,1,\dots$, and $k=1$, are just the Jucys--Murphy elements. 
\begin{lemma}\label{shrink}
For $i=1,2,\dots,$   $z_ie_1=e_1z^{(3)}_{i-2}+p_ie_1$, and $T_{i+1}e_1=e_1T^{(3)}_{i-1}-p_{i-1}e_1$. 
\end{lemma}
\begin{proof}
Note that $e_1z_2=p_2e_1+e_1z_1^{(3)}$ and $e_1T_3=e_1T_1^{(3)}-p_1e_1$, and proceed by induction. If $i\ge 3$, then 
\begin{align*}
T_{i+1}e_1&=-e_iT_ie_1-T_ie_1e_i+e_ie_{i-1}T_ie_1e_i-z_{i-1}e_1-e_iz_{i-2}e_1\\
&=-e_1e_iT_{i-2}^{(3)}-e_1T_{i-2}^{(3)}e_i+2p_{i-2}e_1e_i+e_1e_ie_{i-1}T_{i-2}^{(3)}e_i-p_{i-2}e_1e_{i}e_{i-1}e_{i}\\
&\qquad-e_1z_{i-3}^{(3)}-p_{i-1}e_1-e_1e_iz_{i-4}^{(3)}-p_{i-2}e_1e_i\\
&=-e_1e_iT_{i-2}^{(3)}-e_1T_{i-2}^{(3)}e_i+e_1e_ie_{i-1}T_{i-2}^{(3)}e_i-e_1z_{i-3}^{(3)}-e_1e_iz_{i-4}^{(3)}-p_{i-1}e_1\\
&=e_1T_{i-1}^{(3)}-p_{i-1}e_1,
\end{align*}
while
\begin{align*}
z_ie_1=xz_{i-1}e_1+T_ie_1&= xe_1z_{i-3}^{(3)}+xp_ie_1+e_1T_{i-1}^{(3)}-p_{i-1}e_1\\
&=e_1(xz_{i-3}^{(3)}+T_{i-2}^{(3)})+(xp_{i-1}-p_{i-2})e_1=e_1z_{i-2}^{(3)}+p_ie_1,
\end{align*}
as required.
\end{proof}
\begin{corollary}\label{cor:1}
For $i=1,2,\dots$, and $k=1,2,\dots$,  
\begin{align*}
z_i^{(k)}e_k=e_kz^{(k+2)}_{i-2}+p_ie_k,&&\text{and}&&T_{i+1}^{(k)}e_k=e_kT^{(k+2)}_{i-1}-p_{i-1}e_k.
\end{align*}
\end{corollary}
The following elementary lemma will be used to give the eigenvalues of the Jucys--Murphy elements. 
\begin{lemma}
The sequence $(p_i\in R:i=0,1,\dots)$ satisfies the following relations:
\begin{align}
p_{2j}+p_{2j+2}+\cdots+p_{2i}&=p_{i-k+1}p_{i+k}-p_{j-k}p_{j+k+1},&&\text{for $k\in\{1,2,\dots,j\}$;}\label{stat:1}\\
p_{2j+1}+p_{2j+3}+\cdots+p_{2i+1}&=p_{i-k+1}p_{i+k+1}-p_{j-k}p_{j+k},&&\text{for $k\in\{0,1,\dots,j\}$.}\label{stat:2} 
\end{align}
\end{lemma}
\begin{proof}
We first show that, 
\begin{align}
p_{2i}&=p_{i-k+1}p_{i+k}-p_{i-k}p_{i+k-1},&&\text{for $k=1,\dots,i$.}\label{stat:0}\\
p_{2i+1}&=p_{i-k+1}p_{i+k+1}-p_{i-k}p_{i+k},&&\text{for $k=0,\dots,i$.}\label{stat:0b}
\end{align}
Since~\eqref{stat:0} and~\eqref{stat:0b} both hold when $k=i$, by induction,
\begin{align*}
p_{i-k+2}p_{i+k-1}-p_{i-k+1}p_{i+k-2}&=(xp_{i-k+1}-p_{i-k})p_{i+k-1}-p_{i-k+1}p_{i+k-2}\\
&=p_{i-k+1}(xp_{i+k-1}-p_{i-k-2})-p_{i-k}p_{i+k-1}\\
&=p_{i-k+1}p_{i+k}-p_{i-k}p_{i+k-1}=p_{2i},
\intertext{and,}
p_{i-k+2}p_{i+k}-p_{i-k+1}p_{i+k-1}&=(xp_{i-k+1}-p_{i-k})p_{i+k}-p_{i-k+1}p_{i+k-1}\\
&=p_{i-k+1}(xp_{i+k}-p_{i-k-1})-p_{i-k}p_{i+k}\\
&=p_{i-k+1}p_{i+k+1}-p_{i-k}p_{i+k}=p_{2i+1}.
\end{align*}
From~\eqref{stat:0}, we have, for $k\in\{1,\dots,j\}$, a telescoping sum 
\begin{multline*}
p_{2j}+p_{2j+2}+\cdots+p_{2i}=p_{j-k+1}p_{j+k}-p_{j-k}p_{j+k-1}\\
+p_{j-k+2}p_{j+k+1}-p_{j-k+1}p_{j+k}+\cdots+p_{i-k+1}p_{i+k}-p_{i-k}p_{i+k-1}\\
=p_{i-k+1}p_{i+k}-p_{j-k}p_{j+k-1}.
\end{multline*}
Similarly, from~\eqref{stat:0b}, we have, for $k\in\{0,\dots,j\}$, a sum
\begin{multline*}
p_{2j+1}+p_{2j+3}+\cdots+p_{2i+1}=p_{j-k+1}p_{j+k+1}-p_{j-k}p_{j+k}+\\
+p_{j-k+2}p_{j+k+2}-p_{j-k+1}p_{j+k+1}+\cdots+p_{i-k+1}p_{i+k+1}-p_{i-k}p_{i+k}\\
=p_{i-k+1}p_{i+k+1}-p_{j-k}p_{j+k}.
\end{multline*} 
\end{proof}
It will be useful to note that from the previous lemma:
\begin{align*}
&p_4+p_6+\cdots+p_{2i}=p_{i-1}p_{i+2}=(xp_i-p_{i+1})p_{i+2};\\
&p_{2j}+p_{2j+2}+\cdots+p_{2i}=p_{i-j+1}p_{i+j}&&\text{for $j=1,2,\dots,i-1$};\\
&p_{2j+1}+p_{2j+3}+\cdots+p_{2i+1}=p_{i-j+1}p_{i+j+1}&&\text{for $j=0,1,\dots,i-1$.}
\end{align*}
\begin{lemma}
If $\lambda=(f,n-2f)$ is a partion of $n$, and $i$ is an integer, $1\le i\le n$, then
\begin{align*}
m_{\mathfrak{t}^\lambda}z_i=
\begin{cases}
p_{f}p_{i-f+1}m_{\mathfrak{t}^\lambda},&\text{if $2f+1\le i \le n$;}\\
p_{k}p_{i-k+1}m_{\mathfrak{t}^\lambda},&\text{if $i=2k$, or $i=2k+1$, and $0\le k\le f$.}
\end{cases}
\end{align*}
\end{lemma}
\begin{proof}
The statement being true when $i\le 1$, we proceed by induction, first considering the case where $i\in\{2k,2k+1:0\le k\le f\}$. If $2=k\le f$ and and $z_{2k-1}$ acts on $m_{\mathfrak{t}^\lambda}$ by the scalar $p_{k-1}p_{k+1}$, then
\begin{align*}
m_\lambda T_{2k}=e_1e_3 T_{2k}=e_1e_{3}T_{2}^{(3)}-p_{2}e_{1}e_{3}=0,
\end{align*}
so that, in this instance, $z_{2k}=(xz_{2k-1}+T_{2k})$ acts by $xp_{k-1}p_{k+1}=p_kp_{k+1}$. If $2<k\le f$, and $z_{2k-1}$ acts on $m_{\mathfrak{t}^\lambda}$ by the scalar $p_{k-1}p_{k+1}$, then
\begin{align*}
e_1e_3\cdots e_{2k-1}T_{2k}&=e_1e_3\cdots e_{2k-1}T_{4}^{(2k-3)}-(p_4+p_6+\cdots+p_{2k-2})e_1e_3\cdots e_{2k-1}\\
&=-(p_4+p_6+\cdots+p_{2k-2})e_1e_3\cdots e_{2k-1}\\
&=(p_{k}p_{k+1}-xp_{k-1}p_{k+1})e_1e_3\cdots e_{2k-1},
\end{align*}
so that
\begin{align*}
m_{\mathfrak{t}^\lambda}z_{2k}=m_{\mathfrak{t}^\lambda}(xz_{2k-1}+T_{2k})=(xp_{k-1}p_{k+1}+p_{k}p_{k+1}-xp_{k-1}p_{k+1})m_{\mathfrak{t}^\lambda}=p_{k}p_{k+1}m_{\mathfrak{t}^\lambda}.
\end{align*}
If $0\le k\le f$, and $z_{2k}$ acts on $m_{\mathfrak{t}^\lambda}$ by the scalar $p_kp_{k+1}$, then
\begin{align*}
e_1e_3\cdots e_{2k-1}T_{2k+1}&=e_1e_3\cdots e_{2k-1}T_{1}^{(2k+1)}-(p_1+p_3+\cdots+p_{2k-1})e_1e_3\cdots e_{2k-1}\\
&=(p_kp_{k+2}-xp_kp_{k+1})e_1e_3\cdots e_{2k-1},
\end{align*}
so that
\begin{align*}
m_{\mathfrak{t}^\lambda}z_{2k+1}=m_{\mathfrak{t}^\lambda}(xz_{2k}+T_{2k+1})=(xp_{k}p_{k+1}+p_kp_{k+2}-xp_{k}p_{k+1})m_{\mathfrak{t}^\lambda}=p_kp_{k+2}m_{\mathfrak{t}^\lambda}.
\end{align*}

Now we turn our attention to the action of $z_i$ case where $2f+1\le i\le n$. If $i=2k$, where $f<k$, and $z_{i-1}$ acts on $m_{\mathfrak{t}^\lambda}$ by the scalar $p_fp_{2k-f}$, then 
\begin{align*}
m_\lambda T_{2k}&=e_1e_3\cdots e_{2f-1}T_{2k}=e_1e_3\cdots e_{2f-1}T_{2k-2f}^{(2f+1)}-\sum_{j=k-f}^{k-1}p_{2j}e_1e_3\cdots e_{2f-1}\\
&\equiv -\sum_{j=k-f}^{k-1}p_{2j}m_\lambda=-p_{f}p_{2k-f-1}m_\lambda \mod{\check{\mathcal{A}}_n^\lambda},
\end{align*}
so $z_{2k}$ acts on $m_{\mathfrak{t}^\lambda}$ by the scalar 
\begin{align*}
xp_fp_{2k-f}-p_fp_{2k-f-1}=p_f(xp_{2k-f}-p_{2k-f-1})=p_fp_{2k-f+1}=p_fp_{i-f+1}.
\end{align*}
Similarly, if $i=2k+1$, and $z_{i-1}$ acts on $m_{\mathfrak{t}^\lambda}$ by the scalar $p_fp_{2k-f+1}$ then 
\begin{align*}
m_\lambda T_{2k+1}&=e_1e_3\cdots e_{2f-1}T_{2k+1}=e_1e_3\cdots e_{2f-1}T_{2k-2f+1}^{(2f+1)}-\sum_{j=k-f}^{k-1}p_{2j+1}e_1e_3\cdots e_{2f-1}\\
&\equiv -\sum_{j=k-f}^{k-1}p_{2j+1}m_\lambda=-p_{f}p_{2k-f}m_\lambda \mod{\check{\mathcal{A}}_n^\lambda},
\end{align*}
so $z_{2k+1}$ acts on $m_{\mathfrak{t}^\lambda}$ by the scalar 
\begin{align*}
xp_fp_{2k-f+1}-p_fp_{2k-f}=p_f(xp_{2k-f+1}-p_{2k-f})=p_fp_{2k-f+2}=p_fp_{i-f+1}.
\end{align*}
\end{proof} 
Let $\lambda$ be a partition of $n$ and $\mathfrak{t}\in\mathfrak{T}_n(\lambda)$.  Define a sequence $(r_\mathfrak{t}(k)\in R:k=0,1,\dots,n)$ by $r_\mathfrak{t}(0)=0$, $r_\mathfrak{t}(1)=0$, and 
\begin{align*}
r_{\mathfrak{t}}(k)=
\begin{cases}
(p_i-xp_{i-1})p_{k-i+1}, &\text{if $\SHAPE(\mathfrak{t}|_k)=(i,k-2i)$,}\\
&\qquad\text{and $\SHAPE(\mathfrak{t}|_{k-1})=(i-1,k-2i+1)$;}\\
(p_{k-i+1}-xp_{k-i})p_i,&\text{if $\SHAPE(\mathfrak{t}|_k)=(i,k-2i)$,}\\
 &\qquad\text{and $\SHAPE(\mathfrak{t}|_{k-1})=(i,k-2i-1)$,}
\end{cases}
\end{align*}
for $k=2,3,\dots,n$.
\begin{corollary}
Let $\lambda$ be a partition of $n$ and $\mathfrak{t}\in\mathfrak{T}_{n}(\lambda)$. If $k$ is an integer, $0\le k\le n$, then there exist $a_\mathfrak{v}\in R$, for $\mathfrak{v}\in\mathfrak{T}_n(\lambda)$, such that 
\begin{align*}
m_\mathfrak{t} T_k=r_\mathfrak{t}(k)m_\mathfrak{t}+\sum_{\substack{\mathfrak{v}\rhd\mathfrak{t}}}a_\mathfrak{v}m_\mathfrak{v}.
\end{align*}
\end{corollary}
If $\lambda$ is a partition of $n$ and $\mathfrak{t}\in\mathfrak{T}_n(\lambda)$, define
\begin{align*}
z_\mathfrak{t}(k)=\sum_{i=0}^k x^{i}r_\mathfrak{t}(k-i),&&\text{for $k=0,1,\dots,n$.}
\end{align*}
\begin{corollary}
Let $\lambda$ be a partition of $n$ and $\mathfrak{t}\in\mathfrak{T}_{n}(\lambda)$. If $\mu=(i,k-2i)$ and $\SHAPE(\mathfrak{t}|_k)=\mu$, then $z_\mathfrak{t}(k)=p_ip_{k-i+1}$, and there exist $a_\mathfrak{s}\in R$, for $\mathfrak{s}\in\mathfrak{T}_n(\lambda)$, such that 
\begin{align*}
m_\mathfrak{t}z_k=z_{\mathfrak{t}}(k)m_\mathfrak{t}+\sum_{\mathfrak{s}\rhd\mathfrak{t}}a_\mathfrak{s}m_\mathfrak{s}.
\end{align*}
\end{corollary}
\begin{lemma}\label{jm:diffs}
Let $\lambda$ be a partition of $n$ and $\mathfrak{s},\mathfrak{t}\in\mathfrak{T}_n(\lambda)$. If $\mathfrak{s}|_{n-1}=\mathfrak{t}|_{n-1}$ and $r_\mathfrak{s}(n)=r_\mathfrak{t}(n)$, then $\mathfrak{s}=\mathfrak{t}$.
\end{lemma}
\section{An Orthogonal Basis}
To determine the action of the generators of $\mathcal{A}_n(x)$ on an orthogonal basis for $\mathcal{A}_n(x)$, we require an alternative definition of the Jucys--Murphy elements as given in Corollary~\ref{cor:newdef}.
\begin{lemma}
For $i=2,3,\dots$, 
\begin{align*}
e_{i}e_{i-1}T_ie_i=-xz_{i-1}e_i+z_{i-2}e_i+p_ie_i,&&\text{and,}&& e_iT_ie_i=-2z_{i-1}e_i+p_{i-1}e_i.
\end{align*}
\end{lemma}
\begin{proof}
The lemma being true when $i=2$, we proceed by induction:
\begin{align*}
&e_{i+1}e_{i}T_{i+1}e_{i+1}=-xe_{i+1}e_iT_ie_{i+1}-e_{i+1}e_iT_ie_ie_{i+1}+xe_{i+1}e_ie_{i-1}T_ie_{i}e_{i+1}-z_{i-1}e_{i+1}-xe_{i+1}z_{i-2}\\
&=-xT_ie_{i+1}+2z_{i-1}e_{i+1}-p_{i-1}e_{i+1}-x^2z_{i-1}e_{i+1}+xz_{i-2}e_{i+1}+xp_ie_{i+1}-z_{i-1}e_{i+1}-xz_{i-2}e_{i+1}\\
&=-xz_{i}e_{i+1}+z_{i-1}e_{i+1}+p_{i+1}e_{i+1},
\end{align*}
while
\begin{align*}
e_{i+1}T_{i+1}e_{i+1}&=-2T_ie_{i+1}+e_{i+1}e_ie_{i-1}T_ie_ie_{i+1}-xz_{i-1}e_{i+1}-z_{i-2}e_{i+1}\\
&=-2T_ie_{i+1}-xz_{i-1}e_{i+1}+z_{i-2}e_{i+1}+p_ie_{i+1}-xz_{i-1}e_{i+1}-z_{i-2}e_{i+1}\\
&=-2z_{i}e_{i+1}+p_{i}e_{i+1}.
\end{align*}
\end{proof}
\begin{corollary}\label{cor:newdef}
For $i=2,3,\dots$,
\begin{align*}
T_{i+1}=-e_iT_i-T_ie_i+(p_i-xz_{i-1})e_i-z_{i-1}.
\end{align*}
\end{corollary}
\begin{corollary}
For $i=2,3,\dots$, and $k=1,2,\dots$,
\begin{align}
e_iT_{i+1}^k=e_{i}(p_{i+1}-xz_{i}+z_{i-1})^k\label{msr:1}
\end{align}
\end{corollary}
\begin{proof}
From Corollary~\eqref{cor:newdef},
\begin{align*}
e_iT_{i+1}+xe_iT_i&=-e_iT_ie_i-x^2z_{i-1}e_i-z_{i-1}e_i+xp_ie_i\\
&=2z_{i-1}e_i-p_{i-1}e_i-x^2z_{i-1}e_i-z_{i-1}e_i+xp_ie_i,
\end{align*}
which shows that the statement~\eqref{msr:1} is true when $k=1$. The general case now follows. 
\end{proof}

In what follows, we let $F$ denote the field of fractions of $R$, and denote by $\mathscr{A}_n(x)$ the $F$--algebra generated by $e_1,\dots,e_n$, so that $\mathscr{A}_n(x)=\mathcal{A}_n(x)\otimes_R F$. Following~\cite{djmurphy:97}, let
\begin{align*}
\mathfrak{R}(k)=\{r_\mathfrak{t}(k):\text{$\mathfrak{t}\in\mathfrak{T}_k(\lambda)$ for $\lambda$ a partion of $n$}\}.
\end{align*}
Let $\lambda$ be a partition of $n$. As in~\cite{djmurphy:97}, if $\mathfrak{t}\in\mathfrak{T}_n(\lambda)$, define 
\begin{align*}
F_\mathfrak{t}=\prod_{k=2}^n\,\prod_{\substack{r\in\mathfrak{R}(k),\\r\ne r_\mathfrak{s}(k)}}\frac{ T_k-r }{r_\mathfrak{s}(k)-r},
\end{align*}
and let $f_\mathfrak{t}=m_\mathfrak{t}F_\mathfrak{t}$.
\begin{lemma}
If $\lambda$ is a partition of $n$, then
\begin{enumerate}
\item if $\mathfrak{t}\in\mathfrak{T}_n(\lambda)$, then there exist $a_\mathfrak{u}\in F$, for $\mathfrak{u}\in\mathfrak{T}_n(\lambda)$, such that
\begin{align*}
f_\mathfrak{t}=m_\mathfrak{t}+\sum_{\mathfrak{u}\rhd\mathfrak{t}}a_\mathfrak{u}m_\mathfrak{u};
\end{align*}
\item the set $\{f_\mathfrak{t} : \mathfrak{t}\in\mathfrak{T}_n(\lambda)\}$ is a basis over $F$ for the $\mathscr{A}_n(x)$--module $C^\lambda$;
\item if $\mathfrak{s},\mathfrak{t}\in\mathfrak{T}_n(\lambda)$, then $\langle f_\mathfrak{s},f_\mathfrak{t}\rangle=\delta_{\mathfrak{s},\mathfrak{t}}$;
\item if $\mathfrak{t}\in\mathfrak{T}_n(\lambda)$ and $k$ is an integer, $0\le k\le n$, then $f_\mathfrak{t}T_k=r_\mathfrak{t}(k)f_\mathfrak{t}$ and $f_\mathfrak{t}z_k=z_\mathfrak{t}(k)f_\mathfrak{t}$.
\end{enumerate}
\end{lemma}

Let $\lambda$ be a partition of $n$. If $\mathfrak{t}\in\mathfrak{T}_n(\lambda)$ and $k$ is an integer $1\le k\le n$, define a set $\{e_k(\mathfrak{s},\mathfrak{t})\in F:\mathfrak{s}\in\mathfrak{T}_n(\lambda)\}$ by 
\begin{align}\label{ekst:def}
f_\mathfrak{t}e_k=\sum_{\mathfrak{s}\in\mathfrak{T}_n(\lambda)}e_k(\mathfrak{s},\mathfrak{t})f_\mathfrak{s}.
\end{align}
\begin{lemma}\label{lem:snorm}
Let $\lambda$ be a partition of $n$ and $\mathfrak{t}\in\mathfrak{T}_n(\lambda)$. If $e_k(\mathfrak{s},\mathfrak{t})$ are determined by~\eqref{ekst:def}, then 
\begin{enumerate}
\item if $\mathfrak{s}\in\mathfrak{T}_n(\lambda)$ and $e_k(\mathfrak{s},\mathfrak{t})\ne 0$, then $r_\mathfrak{s}(i)=r_\mathfrak{t}(i)$ whenever $i\ne k-1$ and $i\ne k$;\label{snorm:1}
\item if $e_k(\mathfrak{s},\mathfrak{t})\ne 0$ for some $\mathfrak{s}\in\mathfrak{T}_n(\lambda)$, then $\mathfrak{t}^{(k-1)}=\mathfrak{t}^{(k+1)}$ and $\mathfrak{s}^{(k-1)}=\mathfrak{s}^{(k+1)}$;\label{snorm:2}
\item if $e_k(\mathfrak{s},\mathfrak{t})\ne 0$ for some $\mathfrak{s}\in\mathfrak{T}_n(\lambda)$ with $\mathfrak{s}\ne\mathfrak{t}$, then $e_k(\mathfrak{u},\mathfrak{t})=0$ whenever $\mathfrak{u}\not\in\{\mathfrak{s},\mathfrak{t}\}$;\label{snorm:3}
\item if $e_k(\mathfrak{s},\mathfrak{t})\ne 0$, where $\mathfrak{s}\in\mathfrak{T}_n(\lambda)$, and $\mathfrak{t}\rhd\mathfrak{s}$, then \label{snorm:4}
\begin{align*}
e_k(\mathfrak{t},\mathfrak{t})=\frac{r_\mathfrak{t}(k+1)+z_\mathfrak{t}(k-1)}{p_k-xz_\mathfrak{t}(k-1)-2r_\mathfrak{t}(k)},&&\text{and}&&e_k(\mathfrak{s},\mathfrak{t})&=1;\\
e_k(\mathfrak{s},\mathfrak{s})=\frac{r_\mathfrak{s}(k+1)+z_\mathfrak{s}(k-1)}{p_k-xz_\mathfrak{s}(k-1)-2r_\mathfrak{s}(k)},&&\text{and}&&e_k(\mathfrak{t},\mathfrak{s})&=e_k(\mathfrak{s},\mathfrak{s})e_k(\mathfrak{t},\mathfrak{t}).
\end{align*}
\end{enumerate}
\end{lemma}
\begin{proof}
The item~(\ref{snorm:1}) follows from the fact that $e_k$ commutes with $T_i$ whenever $i\not\in\{k-1,k\}$. For the item~(\ref{snorm:2}), we use the fact (\emph{cf.} Lemma~\ref{shrink}) that
\begin{align*}
e_kz_{k+1}=e_k(z_{k-1}+p_{k+1})
\end{align*}
to observe that 
\begin{align}\label{sn:bsides}
z_\mathfrak{t}(k+1)\sum_{\mathfrak{s}\in\mathfrak{T}_n(\lambda)}e_k(\mathfrak{s},\mathfrak{t})f_\mathfrak{s}=\sum_{\mathfrak{s}\in\mathfrak{T}_n(\lambda)}(z_\mathfrak{s}(k-1)+p_{k+1})f_\mathfrak{s}.
\end{align}
If $\mathfrak{s}=(\lambda^{(0)},\dots,\lambda^{(n)})$ and $\SHAPE(\mathfrak{s}|_{k-1})=(i-1,k+1-2i)$, then there are four possibilities for the sequence $(\lambda^{(k-1)},\lambda^{(k)},\lambda^{(k+1)})$ given as follows:
\begin{align}
(\lambda^{(k-1)},\lambda^{(k)},\lambda^{(k+1)})&=((i-1,k-2i+1),(i,k-2i),(i,k-2i+1))\label{sn:case:1}\\
(\lambda^{(k-1)},\lambda^{(k)},\lambda^{(k+1)})&=((i-1,k-2i+1),(i-1,k-2i+2),(i,k-2i+1))\label{sn:case:2}\\
(\lambda^{(k-1)},\lambda^{(k)},\lambda^{(k+1)})&=((i-1,k-2i+1),(i,k-2i),(i+1,k-2i-1))\label{sn:case:3}\\
(\lambda^{(k-1)},\lambda^{(k)},\lambda^{(k+1)})&=((i-1,k-2i+1),(i,k-2i+2),(i,k-2i+3)).\label{sn:case:4}
\end{align}
Taking $\mathfrak{s}\in\mathfrak{T}_n(\lambda)$ as in~\eqref{sn:case:3} and \eqref{sn:case:4} respectively, the expression~\eqref{sn:bsides} gives:
\begin{align}
&e_k(\mathfrak{s},\mathfrak{t})p_{i+1}p_{k-i+1}=e_k(\mathfrak{s},\mathfrak{t})(p_{i-1}p_{k-i+1}+p_{k+1});\label{sn:case:a}\\
&e_k(\mathfrak{s},\mathfrak{t})p_{i-1}p_{k-i+2}=e_k(\mathfrak{s},\mathfrak{t})(p_{i-1}p_{k-i+1}+p_{k+1});\label{sn:case:b}
\end{align}
Given that
\begin{align*}
p_{k+1}=p_{i+1}p_{k-i+2}-p_{i-1}p_{k-i+1},&&\text{for $k=1,2,\dots$, and $i=1,\dots,k$,}
\end{align*}
the statements~\eqref{sn:case:a} and~\eqref{sn:case:b} imply respectively that 
\begin{align*}
p_{i+1}(p_{k-i+2}-p_{k-i+1})e_k(\mathfrak{s},\mathfrak{t})=0,&&\text{and}&&p_{k-i+1}(p_{i+1}-p_i)e_k(\mathfrak{s},\mathfrak{t})=0,
\end{align*}
conclusions which are patently absurd unless $e_k(\mathfrak{s},\mathfrak{t})=0$ whenever $\mathfrak{s}^{(k-1)}\ne\mathfrak{s}^{(k+1)}$. The statement~(\ref{snorm:3}) now follows. 

For the statement~(\ref{snorm:4}), we have
\begin{align*}
f_\mathfrak{t}T_{k+1}=-f_\mathfrak{t}e_kT_k-r_\mathfrak{t}(k)f_\mathfrak{t}e_k+(p_k-xz_\mathfrak{t}(k-1))f_\mathfrak{t}e_k-z_\mathfrak{t}(k-1)f_\mathfrak{t}
\end{align*}
which, comparing the coefficient of $f_\mathfrak{t}$ on both sides, implies that 
\begin{align*}
r_\mathfrak{t}(k+1)=-r_\mathfrak{t}(k)e_k(\mathfrak{t},\mathfrak{t})-r_\mathfrak{t}(k)e_k(\mathfrak{t},\mathfrak{t})+(p_k-xz_\mathfrak{t}(k-1))e_k(\mathfrak{t},\mathfrak{t})-z_\mathfrak{t}(k-1).
\end{align*}
Now the fact that $r_\mathfrak{t}(k+1)+z_\mathfrak{t}(k-1)\ne0$ shows that the stated expression for $e_k(\mathfrak{t},\mathfrak{t})$ holds; the same reasoning yields the stated expression for $e_k(\mathfrak{s},\mathfrak{s})$. To observe that $e_k(\mathfrak{s},\mathfrak{t})=1$, by the maximality property of $\mathfrak{t}$ and Corollary~\ref{restr}, 
\begin{align*}
f_\mathfrak{t}e_k=m_\mathfrak{t}e_k+\sum_{\mathfrak{u}\rhd\mathfrak{t}}a_\mathfrak{u}f_\mathfrak{u}e_k=m_\mathfrak{t}e_k=m_\mathfrak{s}=f_\mathfrak{s}-\sum_{\mathfrak{v}\rhd\mathfrak{s}}\alpha_\mathfrak{v}f_\mathfrak{v},
\end{align*}
for some $a_\mathfrak{u},\alpha_\mathfrak{v}\in F$, with $\mathfrak{u},\mathfrak{v}\in\mathfrak{T}_n(\lambda)$. 

To complete the proof of the lemma, $f_\mathfrak{t}e_k=e_k(\mathfrak{t},\mathfrak{t})f_\mathfrak{t}+f_\mathfrak{s}$ implies that
\begin{align*}
xf_\mathfrak{t}e_k= f_\mathfrak{t}e_k^2=e_{k}(\mathfrak{t},\mathfrak{t})(e_{k}(\mathfrak{t},\mathfrak{t})f_\mathfrak{t}+f_\mathfrak{s})+e_k(\mathfrak{t},\mathfrak{s})f_\mathfrak{t}+e_k(\mathfrak{s},\mathfrak{s})f_\mathfrak{s},
\end{align*}
whence, comparing coefficients,
\begin{align*}
x=e_k(\mathfrak{t},\mathfrak{t})+e_k(\mathfrak{s},\mathfrak{s}),&&\text{and}&&xe_k(\mathfrak{t},\mathfrak{t})=(e_k(\mathfrak{t},\mathfrak{t}))^2+e_k(\mathfrak{t},\mathfrak{s}).
\end{align*}
Thus
\begin{align*}
e_k(\mathfrak{t},\mathfrak{s})=e_k(\mathfrak{t},\mathfrak{t})(x-e_k(\mathfrak{t},\mathfrak{t}))=e_k(\mathfrak{s},\mathfrak{s})e_k(\mathfrak{t},\mathfrak{t}).
\end{align*}
\end{proof}
\begin{corollary}\label{cor:snorm}
Let $\lambda$ be a partition of $n$, and $k$ be an integer, $1\le k<n$. Suppose that $\mathfrak{s},\mathfrak{t}\in\mathfrak{T}_n(\lambda)$ satisfy $\mathfrak{t}\rhd\mathfrak{s}$ and $e_k(\mathfrak{s},\mathfrak{t})\ne 0$. If $\SHAPE(\mathfrak{t}|_{k})=(i,k-2i)$, then 
\begin{align*}
e_k(\mathfrak{t},\mathfrak{t})=\frac{p_{k-2i+1}}{p_{k-2i+2}},&&\text{and}&&e_k(\mathfrak{s},\mathfrak{s})=\frac{p_{k-2i+3}}{p_{k-2i+2}}.
\end{align*}
\end{corollary}
\begin{proof}
From~\eqref{sn:case:1} and~\eqref{sn:case:2}, we have
\begin{align*}
r_\mathfrak{t}(k)=(p_i-xp_{i-1})p_{k-i+1},&&r_{\mathfrak{t}}(k+1)&=(p_{k-i+2}-xp_{k-i+1})p_i,&&z_{\mathfrak{t}}(k-1)=p_{i-1}p_{k-i+1}\\
r_\mathfrak{s}(k)=(p_{k-i+2}-xp_{k-i+1})p_{i-1},&&r_{\mathfrak{t}}(k+1)&=(p_{i}-xp_{i-i})p_{k-i+2},&&z_{\mathfrak{t}}(k-1)=p_{i-1}p_{k-i+1}.
\end{align*}
Substituting the above into the expressions provided in item~(\ref{snorm:4}) of Lemma~\ref{lem:snorm},
\begin{align*}
e_{k}(\mathfrak{t},\mathfrak{t})&=\frac{(p_{k-i+2}-xp_{k-i+1})p_i+p_{i-1}p_{k-i+1}}{p_k-xp_{i-1}p_{k-i+1}-2(p_i-xp_{i-1})p_{k-i+1} }=\frac{p_ip_{k-i+2}-p_{i+1}p_{k-i+1}}{p_k+xp_{i-1}p_{k-i+1}-2p_ip_{k-i+1}},
\intertext{and}
e_k(\mathfrak{s},\mathfrak{s})&=\frac{(p_{i}-xp_{i-i})p_{k-i+2}+p_{i-1}p_{k-i+1}}{p_k-xp_{i-1}p_{k-i+1}-2(p_{k-i+2}-xp_{k-i+1})p_{i-1}}=\frac{p_ip_{k-i+2}-p_{i-1}p_{k-i+3}}{p_k+xp_{i-1}p_{k-i+1}-2p_{i-1}p_{k-i+2}}.
\end{align*}
The required formulae now follow from elementary considerations, namely, if $2\le 2i\le k$,
\begin{align*}
p_{k-2i+1}=p_{i+1}p_{k-i+1}-p_ip_{k-i+2},&&\text{and}&&p_{k-2i+2}=2p_ip_{k-i+1}-xp_{i-1}p_{k-i+1}-p_k;\\
p_{k-2i+3}=p_ip_{k-i+2}-p_{i-1}p_{k-i+3},&&\text{and}&&p_{k-2i+2}=p_k+xp_{i-1}p_{k-i+1}-2p_{i-1}p_{k-i+2}.
\end{align*}
\end{proof}
\begin{lemma}
Let $\lambda$ be a partition of $n$, and $k$ be an integer $1\le k<n$. Suppose that $\mathfrak{s},\mathfrak{t}\in\mathfrak{T}_n(\lambda)$ satisfy $\mathfrak{t}\rhd\mathfrak{s}$ and $e_k(\mathfrak{s},\mathfrak{t})\ne0$. If $\SHAPE(\mathfrak{t}|_k)=(i,k-2i)$, then 
\begin{align*}
\langle f_\mathfrak{s},f_\mathfrak{s} \rangle=\frac{p_{k-2i+1}(xp_{k-2i+2}-p_{k-2i+1})}{(p_{k-2i+2})^2}\langle f_\mathfrak{t},f_\mathfrak{t}\rangle=\frac{p_{k-2i+1}p_{k-2i+3}}{(p_{k-2i+2})^2}\langle f_\mathfrak{t},f_\mathfrak{t}\rangle.
\end{align*}
\end{lemma}
\begin{proof}
Since $f_\mathfrak{t}e_k=e_k(\mathfrak{t},\mathfrak{t})f_\mathfrak{t}+f_\mathfrak{s}$,
\begin{align*}
\langle f_\mathfrak{t}e_k,f_\mathfrak{t}e_k\rangle=(e_k(\mathfrak{t},\mathfrak{t}))^2\langle f_\mathfrak{t},f_\mathfrak{t}\rangle+\langle f_\mathfrak{s},f_\mathfrak{s}\rangle,
\end{align*}
while associativity of the bilinear form implies that
\begin{align*}
\langle f_\mathfrak{t}e_k,f_\mathfrak{t}e_k\rangle=x\langle f_\mathfrak{t}e_k,f_\mathfrak{t}\rangle=xe_k(\mathfrak{t},\mathfrak{t})\langle f_\mathfrak{t},f_\mathfrak{t}\rangle.
\end{align*}
Thus  $\langle f_\mathfrak{s},f_\mathfrak{s} \rangle=e_k(\mathfrak{t},\mathfrak{t})(x-e_k(\mathfrak{t},\mathfrak{t}))\langle f_\mathfrak{t},f_\mathfrak{t}\rangle$, and the result follows from Corollary~\ref{cor:snorm}.
\end{proof}

\section{The Determinant of the Gram Matrix}
If $\lambda$ is a partition of $n$, let $\dim(\lambda)=\sharp\{\mathfrak{t}:\mathfrak{t}\in\mathfrak{T}_n(\lambda)\}$ and write $\det(\lambda)=\prod_{\mathfrak{t}\in\mathfrak{T}_n(\lambda)}\langle f_\mathfrak{t},f_\mathfrak{t}\rangle$ for the determinant of the Gram matrix associated with $C^\lambda$. The next result is a branching law. 
\begin{lemma}\label{lem:branch}
Let $\lambda=(f,n-2f)$ and $\mu=(f-1,n-2f+1)$ be partitions, where $n-2f\ge 0$ and $f\ge 1$. Then
\begin{align*}
\det(\lambda)=\prod_{\nu\to\lambda}\det(\nu)\cdot\left(\frac{p_{n-2f+2}}{p_{n-2f+1}}\right)^{\dim(\mu)}
\end{align*}
\end{lemma}
\begin{proof}
The result will follow once we show that if $\mathfrak{s}\in\mathfrak{T}_n(\lambda)$ and $\mathfrak{u}=\mathfrak{s}|_{n-1}$, then 
\begin{align}\label{re:stat}
\langle f_\mathfrak{s}, f_\mathfrak{s}\rangle =\langle f_\mathfrak{u},f_\mathfrak{u}\rangle \frac{p_{n-2f+2}}{p_{n-2f+1}},&&\text{whenever $\SHAPE(\mathfrak{u})=\mu$}.
\end{align}
To prove the statement~\eqref{re:stat}, we first consider the case where $\mathfrak{u}=\mathfrak{t}^\mu$. To this purpose define a sequence $(v_{\mathfrak{s}_i}\in\mathfrak{T}_n(\lambda):i=0,\dots,n-2f)$ by $v_{\mathfrak{s}_i}=w_{2f,2f+i}$, for $0\le i\le n-2f$. It follows that $\SHAPE(\mathfrak{s}_{i}|_{2f+i-1})=(f-1,i+1)$ and $\SHAPE(\mathfrak{s}_i|_{2f+i})=(f,i)$, while $\mathfrak{s}_{i}\rhd\mathfrak{s}_{i+1}$ and $e_{2f+i}(\mathfrak{s}_{i+1},\mathfrak{s}_i)\ne0$ for $0\le i<n-2f$. Hence 
\begin{align*}
\langle f_{\mathfrak{s}_{i+1}},f_{\mathfrak{s}_{i+1}}\rangle=\langle f_{\mathfrak{s}_{i}},f_{\mathfrak{s}_{i}}\rangle \frac{p_{i+1}p_{i+3}}{(p_{i+2})^2},
\end{align*}
and
\begin{align*}
\langle f_{\mathfrak{s}_{n-2f}},f_{\mathfrak{s}_{n-2f}}\rangle=\langle f_{\mathfrak{s}_0},f_{\mathfrak{s}_0}\rangle
\frac{p_{1}p_{3}}{(p_{2})^2}\frac{p_{2}p_{4}}{(p_{3})^2}\frac{p_{3}p_{5}}{(p_{4})^2}\cdots \frac{p_{n-2f}p_{n-2f+2}}{(p_{n-2f+1})^2}=\langle f_{\mathfrak{s}_0},f_{\mathfrak{s}_0}\rangle\frac{p_{1}}{p_{2}}\frac{p_{n-2f+2}}{p_{n-2f+1}}.
\end{align*}
Since $\mathfrak{s}=\mathfrak{s}_{n-2f}$, and $\langle f_{\mathfrak{s}_0},f_{\mathfrak{s}_0}\rangle=p_2\langle f_{\mathfrak{u}},f_{\mathfrak{u}}\rangle$, the above verifies~\eqref{re:stat} when $\mathfrak{u}=\mathfrak{t}^\mu$. 

Now suppose that $\mathfrak{s}|_{n-1}=\mathfrak{u}\in\mathfrak{T}_{n-1}(\mu)$, and that $\mathfrak{t}^\mu\rhd\mathfrak{u}$. Then there exists $\mathfrak{v}\in\mathfrak{T}_{n-1}(\mu)$ such that $\mathfrak{v}\rhd \mathfrak{u}$ and $v_\mathfrak{u}=v_\mathfrak{v}e_k$, for some $k$ with $1\le k<n$. If $\mathfrak{t}\in\mathfrak{T}_n(\lambda)$ satisfies $\mathfrak{t}|_{n-1}=\mathfrak{v}$, then $\mathfrak{t}\rhd\mathfrak{s}$ and 
\begin{align*}
\langle f_\mathfrak{s},f_\mathfrak{s}\rangle = \langle f_\mathfrak{t},f_\mathfrak{t}\rangle \frac{p_{k-2i+1}p_{k-2i+3}}{(p_{k-2i+2})^2},&&\text{where $(i,k-2i)=\SHAPE(\mathfrak{t}|_k)$} 
\end{align*}
while, by induction, 
\begin{align*}
\langle f_\mathfrak{t},f_\mathfrak{t}\rangle=\langle f_\mathfrak{v},f_\mathfrak{v}\rangle\frac{p_{n-2f+2}}{p_{n-2f+1}}.
\end{align*}
Since
\begin{align*}
\langle f_\mathfrak{u},f_\mathfrak{u}\rangle=\langle f_\mathfrak{v},f_\mathfrak{v}\rangle \frac{p_{k-2i+1}p_{k-2i+3}}{(p_{k-2i+2})^2},&&\text{where $(i,k-2i)=\SHAPE(\mathfrak{v}|_k)$,}
\end{align*}
it follows that~\eqref{re:stat} holds in general.
\end{proof}
If $\lambda=(i,n-2i)$, and $\mu=(j,n-2j)$ are partitions, with $\mu\rhd\lambda$, define
\begin{align*}
g_{\lambda,\mu}=
\biggl[\frac{p_{n-i-j+1}}{p_{i-j}}\biggr]^{\dim(\mu)}
\end{align*}
Since the dimensions of modules $C^\mu$ are given in terms of certain binomial coefficients~\cite{ramhalv:jb}, the next statement gives closed formulae for the Gram determinants associated with the Temperley--Lieb algebras (\emph{cf}. Corollary~4.7 of~\cite{grahamlehrer:1998}).
\begin{lemma}
Let $\lambda=(f,n-2f)$ be a partition. If $n-2f>0$, then
\begin{align*}
\det(\lambda)=\prod_{\mu\rhd\lambda}g_{\lambda,\mu},
\end{align*}
and, if $n-2f=0$, then $\det(\lambda)=\det(\mu)\cdot x^{\dim(\lambda)}$, where $\mu=(f-1,n-2f+1)$. 
\end{lemma}
\begin{proof}
We first assume that $n-2f>0$. Let
\begin{align*}
\lambda^{(i)}=(f-i,n-2f+2i), &&\text{and} &&\mu^{(i)}=(f-i,n-2f+2i-1),&&\text{for $i=0,1,\dots,f$.}
\end{align*}
Then $\lambda=\lambda^{(0)}\rhd\lambda^{(1)}\rhd\dots\rhd\lambda^{(f)}$ and $\mu=\mu^{(0)}\rhd\mu^{(1)}\rhd\dots\rhd\mu^{(f)}$, while $\mu^{(i)}\to\lambda^{(i)}$ for $i=0,\dots,f$, and $\mu^{(i+1)}\to\lambda^{(i)}$ for $i=0,\dots,f-1$. If $n-2f>1$, then from Lemma~\ref{lem:branch} and induction, 
\begin{align*}
\det(\lambda)&=\det\bigl(\mu^{(0)}\bigr)\det\bigl(\mu^{(1)}\bigr)\Bigl(\frac{p_{n-2f+2}}{p_{n-2f+1}}\Bigr)^{\dim(\mu^{(1)})}\\
&=\prod_{i=1}^f g_{\mu^{(0)},\mu^{(i)}}\cdot\prod_{j=1}^{f-1}g_{\mu^{(1)},\mu^{(j+1)}}\cdot\Bigl(\frac{p_{n-2f+2}}{p_{n-2f+1}}\Bigr)^{\dim(\mu^{(1)})},
\end{align*}
where, for $i=1,\dots,f$, and $j=1,\dots,f-1$, 
\begin{align*}
g_{\mu^{(0)},\mu^{(i)}}=\Bigl(\frac{p_{n-2f+i}}{p_i}\Bigr)^{\dim(\mu^{(i)})}&&\text{and}&&
g_{\mu^{(1)},\mu^{(j+1)}}=\Bigl(\frac{p_{n-2f+j+2}}{p_{j}}\Bigr)^{\dim(\mu^{(j+1)})}.
\end{align*}
Thus, 
\begin{align*}%\label{gram:1}
\det(\lambda)&=\prod_{i=1}^f\Bigl(\frac{p_{n-2f+i}}{p_i}\Bigr)^{\dim(\mu^{(i)})}\cdot\prod_{i=1}^{f-1}\Bigl(\frac{p_{n-2f+i+2}}{p_{i}}\Bigr)^{\dim(\mu^{(i+1)})}\cdot\Bigl(\frac{p_{n-2f+2}}{p_{n-2f+1}}\Bigr)^{\dim(\mu^{(1)})}\\
&=\prod_{i=3}^f\Bigl(\frac{p_{n-2f+i}}{p_i}\Bigr)^{\dim(\mu^{(i)})}\cdot\prod_{i=1}^{f-2}\Bigl(\frac{p_{n-2f+i+2}}{p_{i}}\Bigr)^{\dim(\mu^{(i+1)})}\\
&\qquad\qquad\times\Bigl(\frac{p_{n-2f+2}}{p_1}\Bigr)^{\dim(\mu^{(1)})+\dim(\mu^{(2)})}\Bigl(\frac{p_{n-f+1}}{p_f}\Bigr)^{\dim(\mu^{(f)})}\biggl(\frac{1}{p_2}\biggr)^{\dim(\mu^{(2)})}\\
&=\Bigl(\frac{p_{n-f+1}}{p_f}\Bigr)^{\dim(\mu^{(f)})}\cdot\prod_{i=1}^{f-1}\Bigl(\frac{p_{n-2f+i+1}}{p_i}\Bigr)^{\dim(\mu^{(i)})+\dim(\mu^{(i+1)})}\\
&=\prod_{i=1}^{f}\Bigl(\frac{p_{n-2f+i+1}}{p_i}\Bigr)^{\dim(\lambda^{(i)})},
\end{align*}
where
\begin{align*}
\dim(\lambda^{(i)})=\begin{cases}
\dim\bigl(\mu^{(i)}\bigr),&\text{if $i=f$;}\\
\dim\bigl(\mu^{(i)}\bigr)+\dim\bigl(\mu^{(i+1)}\bigr),&\text{otherwise.}
\end{cases}
\end{align*}
On the other hand, 
\begin{align*}
g_{\lambda,\lambda^{(i)}}=\Bigl(\frac{p_{n-2f+i+1}}{p_i}\Bigr)^{\dim(\lambda^{(i)})}, &&\text{for $i=1,\dots,f$,}
\end{align*}
implies that 
\begin{align*}
\prod_{i=1}^fg_{\lambda,\lambda^{(i)}}=\prod_{i=1}^f\Bigl(\frac{p_{n-2f+i+1}}{p_{i}}\Bigr)^{\dim(\lambda^{(i)})}.
\end{align*}
Now suppose that $n-2f=1$ and, for $i=1,\dots,f-1$, let $\nu^{(i)}=(f-i-1,n+2i-2)$. Then, by induction,
\begin{align*}
\det\bigl(\nu^{(0)}\bigr)=\prod_{i=1}^{f-1}\biggl(\frac{p_{n-2f+i+1}}{p_{i}}\biggr)^{\dim(\nu^{(i)})}&&\text{and}&&\det\bigl(\mu^{(0)}\bigr)=\det\bigl(\nu^{(0)}\bigr)\cdot (p_2)^{\dim(\mu^{(0)})}.
\end{align*}
Further, Lemma~\ref{lem:branch} and induction imply that
\begin{align*}
\det(\lambda)&=\det\bigl(\mu^{(0)}\bigr)\det\bigl(\mu^{(1)}\bigr)\biggl(\frac{p_3}{p_2}\biggr)^{\dim(\mu^{(1)})}\\
&=\prod_{i=1}^{f-1}\biggl(\frac{p_{n-2f+i+1}}{p_{i}}\biggr)^{\dim(\nu^{(i)})}\cdot\prod_{i=1}^{f-1}\biggl(\frac{p_{n-2f+i+2}}{p_i}\biggr)^{\dim(\mu^{(i+1)})}\biggl(\frac{1}{p_2}\biggr)^{\dim(\nu^{(1)})}(p_3)^{\dim(\mu^{(1)})}\\
&=\prod_{i=1}^{f-1}\biggl(\frac{p_{i+2}}{p_{i}}\biggr)^{\dim(\nu^{(i)})}\cdot\prod_{i=1}^{f-1}\biggl(\frac{p_{i+3}}{p_i}\biggr)^{\dim(\mu^{(i+1)})}\biggl(\frac{1}{p_2}\biggr)^{\dim(\nu^{(1)})}(p_3)^{\dim(\mu^{(1)})}.
\end{align*}
In the above expression, the exponent of $p_i$ is $j_i$, where
\begin{align*}
j_i=
\begin{cases}
-\dim\bigl(\nu^{(2)}\bigr)-\dim\bigl(\mu^{(3)}\bigr)-\dim\bigl(\nu^{(1)}\bigr),&\text{if $i=2$;}\\
\dim\bigl(\nu^{(1)}\bigr)-\dim\bigl(\nu^{(3)}\bigr)-\dim\bigl(\mu^{(4)}\bigr)+\dim\bigl(\mu^{(1)}\bigr),&\text{if $i=3$;}\\
\dim\bigl(\nu^{(i-2)}\bigr)-\dim\bigl(\nu^{(i)}\bigr)+\dim\bigl(\mu^{(i-2)}\bigr)-\dim\bigl(\nu^{(i+1)}\bigr),&\text{if $4\le i <f$;}\\
\dim\bigl(\nu^{(i-2)}\bigr)+\dim\bigl(\mu^{(i-2)}\bigr),&\text{if $f\le i\le f+1$;}\\
\dim\bigl(\mu^{(i-2)}\bigr),&\text{if $i=f+2$;}\\
0,&\text{otherwise.}
\end{cases}
\end{align*}
On the other hand, with $n-2f=1$, 
\begin{align*}
\prod_{i=1}^fg_{\lambda,\lambda^{(i)}}=\prod_{i=1}^f\Bigl(\frac{p_{n-2f+i+1}}{p_{i}}\Bigr)^{\dim(\lambda^{(i)})}=\prod_{i=1}^f\Bigl(\frac{p_{i+2}}{p_{i}}\Bigr)^{\dim(\lambda^{(i)})}.
\end{align*}
Since relative positions on the Bratteli diagram associated with $\mathcal{A}_n(x)$ imply that
\begin{align*}
j_i=
\begin{cases}
-\dim\bigl(\lambda^{(i)}\bigr),&\text{if $i=2$;}\\
\dim\bigl(\lambda^{(i-2)}\bigr)-\dim\bigl(\lambda^{(i)}\bigr),&\text{if $3\le i \le f$;}\\
\dim\bigl(\lambda^{(i-2)}\bigr),&\text{if $f<i\le f+2$,}
\end{cases}
\end{align*}
the lemma holds in the case where $n-2f=1$. If $n-2f=0$, the given formula for $\det(\lambda)$ follows directly from Lemma~\ref{lem:branch}.
\end{proof}
\begin{example}
If $n=11$ and $\lambda=(5,1)$, then 
\begin{align*}
\det(\lambda)=\Bigl(\frac{p_7}{p_5}\Bigr)\Bigl(\frac{p_6}{p_4}\Bigr)^{10}\Bigl(\frac{p_5}{p_3}\Bigr)^{44}\Bigl(\frac{p_4}{p_2}\Bigr)^{110}\Bigl(\frac{p_3}{p_1}\Bigr)^{165}.
\end{align*}
To write the above expression as a product in $\mathbb{Z}[x]$, we apply~\eqref{stat:0} with $k=1$, 
\begin{align*}
\det(\lambda)&=p_7\biggl(\frac{p_3(p_4-p_2)}{p_2(p_3-p_1)}\biggr)^{10}(p_5)^{43}(p_3-p_1)^{110}(p_3)^{121}\\
&=p_7(p_3-2p_1)^{10}(p_5)^{43}(p_3-p_1)^{100}(p_3)^{131}.
\end{align*}
\end{example}
\begin{remark}
The above results show that the the Temperley--Lieb algebras, besides being cellular in the sense of~\cite{grahamlehrer}, are equipped with a family of Jucys--Murphy elements satisfying the ``separation condition'' defined by A.~Mathas~\cite{mathas:2006}. In a forthcoming note, we demonstrate a similar construction for the partition algebras of~\cite{ram-halvpart}.
\end{remark}

%###################################
\bibliographystyle{plain}

%##########################################################
%#
%#END DOCUMENT
%#
%##########################################################

\end{document}

%% file: fig1.pstex_t
\begin{picture}(0,0)%
\includegraphics{fig1.pstex}%
\end{picture}%
\setlength{\unitlength}{4144sp}%
\begingroup\makeatletter\ifx\SetFigFont\undefined%
\gdef\SetFigFont#1#2#3#4#5{%
  \reset@font\fontsize{#1}{#2pt}%
  \fontfamily{#3}\fontseries{#4}\fontshape{#5}%
  \selectfont}%
\fi\endgroup%
\begin{picture}(7493,7756)(3136,-8214)
\put(3151,-1681){\makebox(0,0)[lb]{\smash{{\SetFigFont{12}{14.4}{\familydefault}{\mddefault}{\updefault}{\color[rgb]{0,0,0}$\mathfrak{s}=((0,0),(0,1),(0,2),(1,1),(2,0),(2,1),(2,2))$:\hspace{2.0em} $m_\lambda v_\mathfrak{s}=$}%
}}}}
\put(3151,-2581){\makebox(0,0)[lb]{\smash{{\SetFigFont{12}{14.4}{\familydefault}{\mddefault}{\updefault}{\color[rgb]{0,0,0}$\mathfrak{s}=((0,0),(0,1),(1,0),(1,1),(1,2),(2,1),(2,2))$:\hspace{2.0em} $m_\lambda v_\mathfrak{s}=$}%
}}}}
\put(3151,-3481){\makebox(0,0)[lb]{\smash{{\SetFigFont{12}{14.4}{\familydefault}{\mddefault}{\updefault}{\color[rgb]{0,0,0}$\mathfrak{s}=((0,0),(0,1),(0,2),(1,1),(1,2),(2,1),(2,2))$:\hspace{2.0em} $m_\lambda v_\mathfrak{s}=$}%
}}}}
\put(3151,-4381){\makebox(0,0)[lb]{\smash{{\SetFigFont{12}{14.4}{\familydefault}{\mddefault}{\updefault}{\color[rgb]{0,0,0}$\mathfrak{s}=((0,0),(0,1),(0,2),(0,3),(1,2),(2,1),(2,2))$:\hspace{2.0em} $m_\lambda v_\mathfrak{s}=$}%
}}}}
\put(3151,-5281){\makebox(0,0)[lb]{\smash{{\SetFigFont{12}{14.4}{\familydefault}{\mddefault}{\updefault}{\color[rgb]{0,0,0}$\mathfrak{s}=((0,0),(0,1),(1,0),(1,1),(1,2),(1,3),(2,2))$:\hspace{2.0em} $m_\lambda v_\mathfrak{s}=$}%
}}}}
\put(3151,-6181){\makebox(0,0)[lb]{\smash{{\SetFigFont{12}{14.4}{\familydefault}{\mddefault}{\updefault}{\color[rgb]{0,0,0}$\mathfrak{s}=((0,0),(0,1),(0,2),(1,1),(1,2),(1,3),(2,2))$:\hspace{2.0em} $m_\lambda v_\mathfrak{s}=$}%
}}}}
\put(3151,-7081){\makebox(0,0)[lb]{\smash{{\SetFigFont{12}{14.4}{\familydefault}{\mddefault}{\updefault}{\color[rgb]{0,0,0}$\mathfrak{s}=((0,0),(0,1),(0,2),(0,3),(1,2),(1,3),(2,2))$:\hspace{2.0em} $m_\lambda v_\mathfrak{s}=$}%
}}}}
\put(3151,-7981){\makebox(0,0)[lb]{\smash{{\SetFigFont{12}{14.4}{\familydefault}{\mddefault}{\updefault}{\color[rgb]{0,0,0}$\mathfrak{s}=((0,0),(0,1),(0,2),(0,3),(0,4),(1,3),(2,2))$:\hspace{2.0em} $m_\lambda v_\mathfrak{s}=$}%
}}}}
\put(3151,-781){\makebox(0,0)[lb]{\smash{{\SetFigFont{12}{14.4}{\familydefault}{\mddefault}{\updefault}{\color[rgb]{0,0,0}$\mathfrak{s}=((0,0),(0,1),(1,0),(1,1),(2,0),(2,1),(2,2))$:\hspace{2.0em} $m_\lambda v_\mathfrak{s}=$}%
}}}}
\end{picture}%